\documentclass[10pt]{article} 
\usepackage[utf8]{inputenc}
\usepackage[top=30mm, left=30mm, right=30mm, bottom=30mm]{geometry}
\usepackage{amsmath,amssymb,amsthm}
\usepackage{dsfont}
\usepackage{color}
\usepackage{lineno,hyperref}
\usepackage{yfonts}
\usepackage{appendix}
\usepackage{babel} 

\newtheorem{theorem}{Theorem}[section]
\newtheorem{lem}[theorem]{Lemma}

\newtheorem{kor}[theorem]{Corollary}
\newtheorem{prop}[theorem]{Proposition}
\newtheorem{rem}[theorem]{Remark}
\theoremstyle{definition}
\newtheorem{dfn}[theorem]{Definition}

\DeclareMathOperator*{\divv}{div}

\DeclareMathOperator*{\curl}{curl}
\DeclareMathOperator*{\Id}{Id}
\DeclareMathOperator*{\loc}{loc}

\newcommand*{\Ascr}{\mathcal A}
\newcommand*{\Bscr}{\mathcal B}

\newcommand*{\Fscr}{\mathcal F}

\newcommand*{\Hscr}{\mathcal H}

\newcommand*{\Lscr}{\mathcal L}

\newcommand*{\Pscr}{\mathcal P}
\newcommand*{\Rscr}{\mathcal R}

\newcommand*{\N}{\mathbb{N}}
\newcommand*{\E}{\mathbb{E}}
\newcommand*{\R}{\mathbb{R}}

\newcommand*{\vrho}{\varrho}

\definecolor{orange}{rgb}{1.0, 0.55, 0.0} 

\newcommand{\footremember}[2]{%
	\footnote{#2}
	\newcounter{#1}
	\setcounter{#1}{\value{footnote}}%
}

\title{On nonlinear Markov processes in the sense of McKean}
\author{%
	Marco Rehmeier\footremember{alley}{Faculty of Mathematics, Bielefeld University, 33615 Bielefeld, Germany. E-mail: mrehmeier@math.uni-bielefeld.de (corresponding author)}%
	\and Michael Röckner\footnote{Faculty of Mathematics, Bielefeld University, 33615 Bielefeld, Germany. E-mail: roeckner@math.uni-bielefeld.de} \footnote{Academy of Mathematics and System Sciences, CAS, Beijing}%
}
\date{\today}

\begin{document}
\maketitle

\begin{abstract}
We study nonlinear time-inhomogeneous Markov processes in the sense of McKean's seminal work \cite{McKean1-classical}. These are given as families of laws $\mathbb{P}_{s,\zeta}$, $s\geq 0$, on path space, where $\zeta$ runs through a set of admissible initial probability measures on $\R^d$. In this paper, we concentrate on the case where every $\mathbb{P}_{s,\zeta}$ is given as the path law of a solution to a McKean--Vlasov SDE, where the latter is allowed to have merely measurable coefficients, which in particular are not necessarily weakly continuous in the measure variable. Our main result is the identification of general and checkable conditions on such general McKean--Vlasov SDEs, which imply that the path laws of their solutions form a nonlinear Markov process. Our notion of nonlinear Markov property is in McKean's spirit, but
more general in order to include processes whose one-dimensional time marginal densities solve a nonlinear parabolic PDE, more precisely, a nonlinear Fokker--Planck--Kolmogorov equation, such as Burgers' equation, the porous media equation and variants thereof with transport-type drift, and also the very recently studied $2D$ vorticity Navier--Stokes equation and the $p$-Laplace equation. In all these cases, the associated McKean--Vlasov SDEs are such that both their diffusion and drift coefficients singularly depend (i.e. Nemytskii-type) on the one-dimensional time marginals of their solutions. We stress that for our main result the nonlinear Fokker--Planck--Kolmogorov equations do not have to be well-posed. Thus, we establish a one-to-one correspondence between solution flows of a large class of nonlinear parabolic PDEs and nonlinear Markov processes.
\end{abstract}
\noindent	\textbf{Keywords:} Nonlinear Fokker--Planck--Kolmogorov equation, distribution-dependent stochastic differential equation, McKean-Vlasov SDE, nonlinear Markov process, Barenblatt solution, porous media equation, superposition principle, probabilistic representation, solution flow, flow selection, extremality \\
\textbf{2020 MSC:} Primary 60H30, 35Q84, 35K55, 60J25; Secondary 60J45, 60J60

\section{Introduction}

In this work we study nonlinear Markov processes in the sense of McKean's seminal work \cite{McKean1-classical}. We suggest a modified definition of such processes (see Definition \ref{def:NL-Markov-process} below), derive sufficient criteria for their construction, and present as well as analyze several classes of examples. These include nonlinear Markov processes associated with the classical porous media equation in all spatial dimensions $d \geq 1$, with their Barenblatt solutions as one-dimensional time marginal densities, and variants thereof with more general diffusivity functions and additional drift of transport type.

\textbf{1.1 Linear situation.}
Let us briefly recall how the theory of classical "linear" Markov processes is linked to stochastic differential equations (SDEs) and linear parabolic Fokker--Planck--Kolmogorov equations (FPKEs).
If
\begin{equation*}
	dX_t = b(t,X_t)dt+\sigma(t,X_t)dB_t, \quad t \geq s,\quad X_s = x,
\end{equation*}
where $B$ is a standard $d$-dimensional Brownian motion, has a unique probabilistic weak solution $X^{s,x}$ for each $(s,x)\in \R_+\times \R^d$, then the corresponding path laws $\mathbb{P}_{s,x}$ form a classical Markov process (see, e.g., \cite{D65} or\cite{S88}), i.e. denoting the natural projections on the space of continuous paths $C([s,\infty),\R^d)$ by $\pi^s_t$, $\pi^s_t(w):= w(t)$, the Markov property
\begin{equation}\label{intro:Markov-prop}\tag{$\ell$MP}
	\mathbb{P}_{s,\zeta}(\pi^s_t \in A| \sigma(\pi^s_\tau, s\leq \tau \leq r)) = \mathbb{P}_{r,\pi^s_r}(\pi^r_t \in A) \quad \mathbb{P}_{s,\zeta}-\text{a.s.} \text{ for all }A \in \Bscr(\R^d)
\end{equation}
holds for all $\zeta \in \Pscr$ (the space of Borel probability measures on $\R^d$) and $0\leq s\leq r \leq t$, where $\mathbb{P}_{s,\zeta}:= \int_{\R^d} \mathbb{P}_{s,x}\,\zeta(dx)$.
By Itô's formula, the one-dimensional time marginals, $\mu^{s,\zeta}_t$, $t \geq s$, of the latter solve the linear FPKE
\begin{equation}\label{intro:lFPE}
	\partial_t \mu_t = \partial^2_{ij}(a_{ij}(t,x)\mu_t)-\partial_i(b_i(t,x)\mu_t),\quad t\geq s, \quad \mu_s = \zeta
\end{equation}
(in the Schwartz distributional sense), where $a = (a_{ij})_{i,j \leq d} = \frac 1 2 \sigma \sigma^T$ and we use Einstein summation convention. In particular, $(\mu^{s,\zeta}_t)_{t \geq s}$ is a continuous curve in $\Pscr$ with the weak topology. Vice versa, if this linear FPKE has a unique weakly continuous probability measure-valued solution for all probability initial conditions
and if the coefficients satisfy a mild integrability condition with respect to these solutions, then by the Ambrosio-Figalli-Trevisan superposition principle \cite{Ambrosio08,Figalli09,Trevisan16} (see also \cite{BRS19-SPpr}) and \cite[Thm.6.2.3]{StroockVaradh2007}, the above SDE is weakly well-posed and the family of path laws $\mathbb{P}_{s,x}$ of its unique weak solutions with Dirac initial conditions is a classical Markov process. 
For a comparison of this linear and our nonlinear case, see Section \ref{sect:comp-to-lin} below.

\textbf{1.2 McKean's vision.}
McKean seems to have been the first who realized that the core Markovian feature, i.e. that the path law $\mathbb{P}_{s,\zeta}$ restricted to $\sigma(\pi^s_\tau, r\leq \tau)$ and conditioned on $\sigma(\pi^s_\tau, s \leq \tau \leq r)$ is a function of $s,r$ and the current position $\pi^s_r$ only, is more general than the specific formula \eqref{intro:Markov-prop} (see also the formulation in \cite[p.145]{StroockVaradh2007}). In \cite{McKean1-classical}, he suggested a generalization by replacing $\mathbb{P}_{r,\pi^s_r}$ by (a regular conditional probability of) $\mathbb{P}_{r,\mu^{s,\zeta}_r}[\,\cdot\,|\pi^r_r = \pi^s_r]$, where $\mu^{s,\zeta}_r$ denotes the one-dimensional time marginal of $\mathbb{P}_{s,\zeta}$ at $r$.
McKean's envisioned program was to connect this more general Markov property to nonlinear parabolic PDEs, more precisely to nonlinear FPKEs of \textit{Nemytskii-type}, i.e., to equations of the form 
\begin{equation}\label{intro:NLFPKE-Nemytskii}
	\partial_t u = \partial^2_{ij}(a_{ij}(t,u,x)u)-\partial_i(b_i(t,u,x)u), \quad t\geq s,\quad u(t,x)dx \xrightarrow{t \to s} \zeta
\end{equation}
by constructing a family of path measures $\mathbb{P}_{s,\zeta}$ satisfying this more general property, whose one-dimensional time marginals have densities which solve such a PDE and, thereby, to provide a natural and rich probabilistic representation for solutions to \eqref{intro:NLFPKE-Nemytskii}. In a later paper, McKean indicated how to construct such families as path laws of solutions to
distribution-dependent SDEs (DDSDEs, also called \textit{McKean--Vlasov equations of Nemytskii-type}) of the form
\begin{equation}\label{intro:DDSDE-N}
	dX_t = b(t,u(t,X_t),X_t)dt+\sigma(t,u(t,X_t),X_t)dB_t, \,\,\, u(t,x)dx = \mathcal{L}_{X_t}, \,\,\, u(t,x)dx\xrightarrow{t \to s}\zeta
\end{equation}
($\mathcal{L}_{X_t} =$ distribution of $X_t$), see \cite{McKean-classical2}.
In addition to special examples, such as discrete state space models from kinetic gas theory and DDSDEs with unit diffusion and drift $b(t,u(t,x),x) = \int_\R (y-x)u(t,y)dy$, he proposed to study very interesting models of type \eqref{intro:NLFPKE-Nemytskii}, such as Burgers' equation and the porous media equation in dimension $d=1$. Though, to the best of our knowledge, he did not go beyond such suggestions in \cite{McKean1-classical} or subsequent works, he was clearly envisioning a much larger theory.

\textbf{1.3 Goals of our paper.}
The aim of this paper is to first present a theory of nonlinear Markov processes based on McKean's suggestions, including a precise definition of such processes, and discuss some of their basic properties. Second, we shall prove a result which specifies sufficient and checkable conditions to construct from solutions to a general \textit{nonlinear} FPKE
\begin{equation}\label{intro:NLFPE}
	\partial_t\mu_t = \partial^2_{ij}(a_{ij}(t,\mu_t,x)\mu_t)- \partial_i(b_i(t,\mu_t,x)\mu_t)
\end{equation}
(of which \eqref{intro:NLFPKE-Nemytskii} is a special case)
nonlinear Markov processes consisting of solution path laws to the associated DDSDE 
\begin{equation}\label{intro:DDSDE}
	dX_t = b(t,\mathcal{L}_{X_t},X_t)dt+\sigma(t,\mathcal{L}_{X_t},X_t)dB_t
\end{equation}
(of which \eqref{intro:DDSDE-N} is a special case)
with one-dimensional time marginals given by these FPKE-solutions. Third, we shall apply this result to a large class of examples.
Our longterm goal is to develop an as rich theory as in the linear case, where the correspondence between linear (i.e. classical) Markov processes and linear PDE theory has become a very versatile and important tool to solve problems about the linear PDE through its corresponding linear Markov process and vice versa.

Our definition of a nonlinear Markov process is in the spirit of \cite{McKean1-classical}, but more general since we allow a possibly restricted class $\Pscr_0 \subseteq \Pscr$ of initial data, see Definition \ref{def:NL-Markov-process} below. This is of high relevance since then our theory also applies to examples in which \eqref{intro:NLFPKE-Nemytskii} can only be solved for initial measures from a restricted class $\Pscr_0$, which is quite common for such nonlinear equations of Nemytskii-type. 


Let us now describe our main result, Theorem \ref{theorem1}.  The starting point is a family of weakly continuous probability measure-valued solutions $\{\mu^{s,\zeta}\}_{(s,\zeta) \in \R_+\times \Pscr_0}$, $\mu^{s,\zeta} = (\mu^{s,\zeta}_t)_{t\geq s}$, $\mu^{s,\zeta}_s = \zeta$, to \eqref{intro:NLFPE} and corresponding solution path laws $\{\mathbb{P}_{s,\zeta}\}_{(s,\zeta) \in \R_+\times \Pscr_0}$ to the associated DDSDE \eqref{intro:DDSDE}  such that 
\begin{equation}\label{intro:marginals}
\mu^{s,\zeta}_t = \mathbb{P}_{s,\zeta}\circ (\pi^s_t)^{-1}
\end{equation}
(such solution path laws exist by the aforementioned superposition principle and its nonlinear extension \cite{BR18_2,BR18}).  Here $\Pscr_0 \subseteq \Pscr$ is an arbitrary prescribed set of admissible initial conditions. Our main result identifies a pair of sufficient conditions implying:
$$\text{ 1) each $\mathbb{P}_{s,\zeta}$ is the unique (!) solution path law with \eqref{intro:marginals}},$$
and 
$$\text{2) $\{\mathbb{P}_{s,\zeta}\}_{(s,\zeta) \in \R_+\times \Pscr_0}$ is a nonlinear Markov process in the sense of Definition \ref{def:NL-Markov-process}.}$$
Hence, we establish a one-to-one correspondence between solution flows for \eqref{intro:NLFPE} and nonlinear Markov processes associated with \eqref{intro:DDSDE}.

Contrary to what one might expect, we do not need the nonlinear FPKE to be well-posed. Instead, our first condition is that the prescribed solutions satisfy the \textit{flow property}, i.e.
\begin{equation}\label{intro:flow-prop}
	\mu^{s,\zeta}_t \in \Pscr_0,\quad\mu^{s,\zeta}_t = \mu^{r,\mu^{s,\zeta}_r}_t,\quad \forall\, 0 \leq s \leq r \leq t,\,\zeta \in \Pscr_0.
\end{equation}
Such solution flows may, for instance, be obtained by selection (\cite{Rehmeier-nonlinear-flow-JDE}) or by their construction via nonlinear semigroup methods (see, e.g., \cite{NLFPK-DDSDE5}) (and, of course, \eqref{intro:flow-prop} is always valid in well-posed cases). The map $\Pscr_0 \ni \zeta \mapsto \mu^{s,\zeta}_t$ is in general not linear on its domain, even when $\Pscr_0$ is a convex set, e.g. $\Pscr_0=\Pscr$. Therefore, we call  Markov processes as in Definition \ref{def:NL-Markov-process} below \textit{nonlinear Markov processes}. 
This nonlinearity is one main difference from the theory of linear Markov processes, where $\Pscr_0 = \Pscr$, and where one has the \textit{Chapman--Kolmogorov equations}, i.e.
$$\mu^{s,\delta_x}_t = \int_{\mathbb{R}^d} \mu^{r,\delta_y}_t \,\mu^{s,\delta_x}_r(dy),\quad \forall\, 0\leq s \leq r\leq t, x \in \R^d,$$
which imply \eqref{intro:flow-prop}. 

Our second condition is a geometric one: each $\mu^{s,\zeta}$ should be an extreme point in the convex set of all weakly continuous probability solutions with initial datum $(s,\zeta)$ to the linearized FPKE obtained by fixing a priori the nonlinear solution $\mu^{s,\zeta}$ in the measure variable of the coefficients in \eqref{intro:NLFPE}.
For the connection between extremality and Markov property for martingale problems in the linear case, see the fundamental paper \cite{SY80}.

Our main result, Theorem \ref{theorem1}, states that under these two conditions, assertions 1) and 2) mentioned above hold. 
 Two main ingredients of the proof are Theorem \ref{aux-prop-ex+uniqu} and Lemma \ref{lem:Michael-observation}. Theorem \ref{aux-prop-ex+uniqu} ensures the uniqueness of $\mathbb{P}_{s,\zeta}$, and Lemma \ref{lem:Michael-observation} gives a useful equivalent formulation of the extremality condition in terms of a restricted uniqueness condition for the corresponding linearized FPKEs. This restricted uniqueness condition is checkable in many applications, see Section \ref{sect:Appl}. We point out that Theorem \ref{aux-prop-ex+uniqu}  and Lemma \ref{lem:Michael-observation} are new results themselves.
 
 An important consequence of our proof is Corollary \ref{cor:after-main-thm}, where we prove that it is sufficient to prove the extremality condition only for solutions $\mu^{s,\zeta}$ with $\zeta$ from a strict subset  $\mathfrak{P}_0$ of all admissible initial conditions $\mathcal{P}_0$, if $\mu^{s,\zeta}_t$ for $t>s$ belongs to $\mathfrak{P}_0$ for all $\zeta \in \Pscr_0$. Then, we still prove 2), and 1) is still valid for all $\zeta \in \mathfrak{P}_0$. This extension is crucial to treat cases (as, for instance, the porous media equation) where $\Pscr_0$ contains Dirac measures, since in many cases it turns out that our extremality condition can only be verified for $\zeta$ which are absolutely continuous w.r.t. Lebesgue measure, compare Section \ref{sect:Appl} below.
 
 Finally, we mention that even for the linear special case, i.e. when the underlying FPKE is of type \eqref{intro:lFPE}, our results yield new outcomes as well, see Section \ref{subsect:lin-special case}.

\textbf{1.4 Applications.}
Our two sufficient conditions mentioned above are, of course, in particular satisfied when the nonlinear and the corresponding linearized FPKEs are well-posed in the class of weakly continuous probability solutions, and we briefly present this case at the beginning of Section \ref{sect:Appl}. Usually, well-posedness for the nonlinear equation is obtained by assuming (Lipschitz-)continuity of the coefficients in their measure variable w.r.t. a Wasserstein distance. 

All our main examples, i.e. Section \ref{subsect:examples-nemytskii} (i)-(iv), are not of this type, but of Nemytskii-type, namely generalized porous media equations with drift in all spatial dimensions $d\geq 1$ (in particular including \textit{nonlinear distorted Brownian motion}, see \cite{BR-IndianaJ,RRW20,BR22-invar-pr}), Burgers' equation, equations of porous media type where the Laplacian is replaced by a fractional Laplacian so that the corresponding DDSDEs have Lévy noise (see also \cite{BR22-frac,BRZ23,SPpr-nonloc}), and the classical porous media equation with its Barenblatt solutions. This is the first work in which the nonlinear Markov property for solutions to the DDSDEs associated with each of these nonlinear PDEs is proven.
In these cases the coefficients are only measurable with respect to the Borel $\sigma$-algebra of the weak topology on $\Pscr$ (see the beginning of Section \ref{subsect:examples-nemytskii}).
As we show in Section \ref{subsect:examples-nemytskii}, we are able to apply the theory developed in this work to all these examples. As already mentioned, in particular the Barenblatt solutions to the porous media equation, $d \geq 1$, uniquely determine a nonlinear Markov process with these solutions as one-dimensional time marginals, cf. Section \ref{subsect:examples-nemytskii} (iv). We consider this as a central new application of our main result.

While, in the spirit of McKean's vision, our focus is on nonlinear FPKEs of Nemytskii-type \eqref{intro:NLFPKE-Nemytskii} and the associated DDSDEs \eqref{intro:DDSDE-N}, as already indicated above, our main result and our examples also cover nonlinear Markov processes with general probability measures as one-dimensional time marginals, arising as solution flows to \eqref{intro:NLFPE}.  In this case, the coefficients of the associated DDSDE \eqref{intro:DDSDE} do not depend pointwise on the one-dimensional time marginal density (which in general need not exist), but on the one-dimensional time marginal as a measure.

Finally, we mention the very recent papers \cite{BRZ23} and \cite{BRR24-pLaplace}, in which an earlier arXiv-version of our work was applied to the $2D$ vorticity Navier--Stokes equation and the $p$-Laplace equation, respectively (see (v) and (vi) in Section \ref{subsect:examples-nemytskii} below). This earlier version was already applied in \cite{BR24-lecturenotes}.

\textbf{1.5 Related literature.}
Concerning the related literature, we would like to mention the very elaborate book by Kolokoltsov \cite{Kolo-book10}, in which nonlinear FPKEs with coefficients with Lipschitz continuous dependence in their measure variable (e.g. with respect to a Wasserstein distance) and their relation to DDSDEs are studied. The book contains well-posedness results for nonlinear FPKEs and their associated linearized equations, possibly including an additional jump operator. For such cases, the author associates with these well-posed equations a family of linear Markov processes in the following way: If for each initial datum $\zeta \in \Pscr$ the linearized FPKE, obtained by fixing the unique solution curve $\mu^\zeta$ to the nonlinear FPKE with initial datum $\zeta$ in the measure component of the coefficients, is well-posed, then there is a unique linear Markov process related to each such linearized equation. More precisely, it is given by the path laws of the unique weak solutions to the stochastic equation associated with this linearized FPKE. The author defines this family of linear Markov processes as a nonlinear Markov process. 
As mentioned before, all our main examples in the present work are of Nemytskii-type  \eqref{intro:NLFPKE-Nemytskii},  and hence
fail to satisfy the assumptions in \cite{Kolo-book10}. Our notion of nonlinear Markov processes, however, is sufficiently general to also cover these cases.

This paper was strongly motivated by the results in \cite{RRW20}, where (among other results) it was shown that uniqueness (in a restricted class) of distributional solutions to nonlinear FPKEs and their linearizations implies weak uniqueness (in a restricted class) of the associated DDSDEs, and as a consequence, the nonlinear Markov property of the path laws of the weak solutions to the latter was shown (see \cite[Cor.4.6]{RRW20}). In comparison with this, our Theorem \ref{theorem1}, does not require such a uniqueness  property, but only the existence of a solution flow satisfying an extremality condition. This enhances the realm of applications substantially (see Chapter \ref{sect:Appl} below) and, in particular, covers the so-called \textit{nonlinear distorted Brownian motion} (NLDBM), which was already studied as a main example in \cite{RRW20}, thus identifying NLDBM as a nonlinear Markov process in the sense of our Definition \ref{def:NL-Markov-process} (see Section \ref{subsect:examples-nemytskii} (i)).

We are not aware of any other paper on nonlinear Markov processes realizing McKean's program after \cite{McKean-classical2,McKean1-classical} in general, or specifically in the Nemytskii-case. However, if one restricts oneself to finding a probabilistic representation for the solutions of \eqref{intro:NLFPKE-Nemytskii} as one-dimensional time marginal laws of a stochastic process without proving that it is a nonlinear Markov process, there are several papers on this in the literature, e.g. \cite{BR12,BR17,BRR10,BRR11,IORT20,LeCOR15,CR14}. Here we would like to draw special attention to the pioneering work \cite{BCRV96}, which treats the classical porous media equation on $\R^1$ on the basis of very nice original ideas and without using the much later discovered technique from \cite{BR18,BR18_2} (see also \cite{BR22-timedep-case}), which in turn are based on the application of the already mentioned superposition principle to the associated linearized porous media equation. As mentioned before, our results in the present paper apply to the classical porous media equation on $\R^d$ for all dimensions $d \geq 1$, see Section \ref{subsect:examples-nemytskii} (iv).

\textbf{1.6 Structure of the paper.}
Section \ref{sect:def-and-basics} contains our definition of nonlinear Markov processes, as well as Proposition \ref{prop:Markov-distr-from-marginals}, which shows that also in the nonlinear case Markovian path laws are uniquely determined by suitably chosen one-dimensional time marginals. A comparison to classical (linear) Markov processes is contained in Subsection \ref{sect:comp-to-lin}. After introducing all relevant equations and definitions of solutions and well-posedness in Subsection \ref{subsect:eq-sol}, we formulate and prove our main results in Subsection \ref{subsect:main-result}. In Subsection \ref{subsect:lin-special case}, we present new results for the classical linear case, which are obtained as special cases of Theorems \ref{aux-prop-ex+uniqu} and \ref{theorem1}. Section \ref{sect:Appl} contains applications of our main result, i.e. we associate nonlinear Markov processes to a large class of nonlinear FPKEs and McKean--Vlasov SDEs.

\subsubsection*{Notation}
The Euclidean norm and inner product on $\R^d$ are denoted by $|\cdot|$ and $\langle \cdot, \cdot \rangle$, and we set $\R_+ := [0,\infty)$. For topological spaces $X$ and $Y$, $C(X,Y)$ denotes the space of continuous functions $f: X \to Y$, $\Bscr(X)$ the Borel $\sigma$-algebra on $X$, and $\Pscr(X)$ the space of all probability measures on $\Bscr(X)$, equipped with the weak topology, i.e. the initial topology of the maps $\Pscr(X) \ni \mu \mapsto \int_Xf(x) \mu(dx)$, for all bounded $f \in C(X,\R)$. We let $\delta_x$ denote the Dirac measure in $x \in X$, and for $\mu \in \Pscr(X)$ and $h:X \to \R$ we sometimes write $\mu(h):= \int_X h(x)\,\mu(dx)$, provided the integral exists. For $X=\R^d$, we write $\Pscr$ instead of $\Pscr(\R^d)$ and $\Pscr_a$ for the space of measures $\mu \in \Pscr$, which are absolutely continuous with respect to Lebesgue measure $dx$, i.e. $\mu \ll dx$. If $\mu \ll \nu$ and $\nu \ll \mu$, we write $\mu \sim \nu$. 

Let $\Bscr_b^{(+)}(\R^d)$ denote the space of bounded (non-negative) Borel functions $f: \R^d \to \R$, and $C_b(\R^d)$ its subspace of continuous functions. For $m \in \N_0$, $C_{(c)}^m(\R^d)$ are the spaces of (compactly supported) continuous functions with continuous partial derivatives up to order $m$, and we write $C_{(c)}(\R^d)$ instead of $C^0_{(c)}(\R^d)$. We use the standard notation $L^p(\R^d;\R^m) = \big\{g: \R^d \to \R^m$ Borel, $|g|_{L^p} < \infty\big\}$ with the usual $L^p$-norms $|\cdot|_{p}$ for $p \in [1,\infty]$. The corresponding local spaces are denoted $L^p_{\loc}(\R^d;\R^m)$. For $m =1$, we write $L^p(\R^d)$ instead of $L^p(\R^d;\R)$.

The path space $C([s,\infty),\R^d)$ with the topology of locally uniform convergence is denoted by $\Omega_s$ with Borel $\sigma$-algebra $\Bscr(\Omega_s) = \sigma(\pi^s_\tau, \tau \geq s)$, where $\pi^s_t, t \geq s,$ are the usual projections on $\Omega_s$. We also use the notation $\Fscr_{s,r} := \sigma(\pi^s_\tau, s \leq \tau \leq r)$ and $\Pi^s_r : \Omega_s\to \Omega_r$, $\Pi^s_r : w \mapsto w_{|[r,\infty)}$ for $s \leq r$.

\section{Nonlinear Markov processes: Definition and comparison to linear case}\label{sect:def-and-basics}
\begin{dfn}\label{def:NL-Markov-process}
	Let $\Pscr_0 \subseteq \Pscr$. A \textit{nonlinear Markov process} is a family $(\mathbb{P}_{s,\zeta})_{(s,\zeta)\in \R_+\times \Pscr_0}$ of probability measures $\mathbb{P}_{s,\zeta}$ on $\Bscr(\Omega_s)$ such that
	\begin{enumerate}
		\item[(i)] The marginals $\mathbb{P}_{s,\zeta}\circ(\pi^s_t)^{-1} =: \mu^{s,\zeta}_t$ belong to $\mathcal{P}_0$ for all $0\leq s \leq t$ and $\zeta \in \Pscr_0$.
		\item[(ii)] The \textit{nonlinear Markov property} holds, i.e. for all $0\leq s \leq r \leq t$, $\zeta \in \Pscr_0$
		\begin{equation}\label{Markov-prop}\tag{MP}
			\mathbb{P}_{s,\zeta}(\pi^{s}_{t} \in A|\mathcal{F}_{s,r})(\cdot) = p_{(s,\zeta),(r,\pi^s_r(\cdot))}(\pi^r_t\in A) \quad \mathbb{P}_{s,\zeta}-\text{a.s.} \text{ for all }A \in \Bscr(\R^d),
		\end{equation}
		where $p_{(s,\zeta),(r,y)}, y \in \R^d$, is a regular conditional probability kernel from $\R^d$ to $\Bscr(\Omega_r)$ of $\mathbb{P}_{r,\mu^{s,\zeta}_r}[\,\,\cdot\, \,| \pi^r_r=y]$, $y \in \R^d$ (i.e. in particular $p_{(s,\zeta),(r,y)} \in \Pscr(\Omega_r)$ and $p_{(s,\zeta),(r,y)}(\pi^r_r = y) = 1$).
	\end{enumerate}
\end{dfn}

The term \emph{nonlinear} Markov property originates from the fact that in the situation of Definition \ref{def:NL-Markov-process} the map $\Pscr_0 \ni \zeta \mapsto \mu^{s,\zeta}_t$ is, in general, not linear on its domain, even if $\Pscr_0$ is a convex set (which we do not assume). For the special case of a classical linear Markov process with $\Pscr_0 = \Pscr$, this map is linear on its domain, see Remark \ref{rem:linear-case-convex}.
\begin{rem}\label{rem:Markov-prop-extends-to-fct}
	\begin{enumerate}
		\item[(i)] Consider the situation of Definition \ref{def:NL-Markov-process}, let $(s,\zeta) \in \R_+\times \Pscr_0$ and $s\leq r$. By definition of $p_{(s,\zeta),(r,y)}, y \in \R^d$, one has
		\begin{equation}\label{aux-eq-rem2.2}
			\mathbb{P}_{r,\mu^{s,\zeta}_r}(\cdot) = \int_{\mathbb{R}^d} p_{(s,\zeta),(r,y)}(\cdot)\,\mu^{s,\zeta}_r(dy) = \int_{\Omega_s} p_{(s,\zeta),(r,\pi^s_r(w))}(\cdot)\,\mathbb{P}_{s,\zeta}(dw).
		\end{equation}
		\item [(ii)] Let $0\leq s \leq r \leq t$ and $\zeta \in \Pscr_0$. By a monotone class-argument, \eqref{Markov-prop} is equivalent to
		$$\mathbb{E}_{s,\zeta}\big[h(\pi^s_t)|\mathcal{F}_{s,r}\big](\cdot) = p_{(s,\zeta),(r,\pi^s_r(\cdot))}\big(h(\pi^r_t)\big)\quad \mathbb{P}_{s,\zeta}-\text{a.s.} \text{ for all }h \in \Bscr_b(\R^d).$$
		\item[(iii)] The one-dimensional time marginals $\mu^{s,\zeta}_t = \mathbb{P}_{s,\zeta}\circ (\pi^s_t)^{-1}$ of a nonlinear Markov process satisfy the \textup{flow property}, i.e.
		\begin{equation*}
			\mu^{s,\zeta}_t = \mu^{r,\mu^{s,\zeta}_r}_t,\quad \forall\, 0\leq s \leq r \leq t, \zeta \in \Pscr_0.
		\end{equation*}
		Indeed, by \eqref{aux-eq-rem2.2}, for all $A\in \Bscr(\R^d)$:
		\begin{align*}
			\mu^{s,\zeta}_t(A) &= \mathbb{E}_{s,\zeta}\big[\mathbb{P}_{s,\zeta}(\pi^s_t \in A|\Fscr_{s,r})\big] \\&= \mathbb{E}_{s,\zeta}\big[p_{(s,\zeta),(r,\pi^s_r)}(\pi^r_t\in A)\big] = \mathbb{P}_{r,\mu^{s,\zeta}_r}\big(\pi^r_t \in A\big) = \mu^{r,\mu^{s,\zeta}_r}_t(A).
		\end{align*}
	\end{enumerate}

\end{rem}
The following proposition shows that the finite-dimensional distributions of the path measures $\mathbb{P}_{s,\zeta}$ of a nonlinear Markov process (and hence the path measures itself) are uniquely determined by the family of one-dimensional time marginals $p_{r,t}^{s,\zeta}(x,dz)$, $s \leq r, x \in \R^d$, defined in \eqref{eq:rcp-marginals} below.
\begin{prop}\label{prop:Markov-distr-from-marginals}
	Let $(\mathbb{P}_{s,\zeta})_{(s,\zeta) \in \R_+\times \Pscr_0}$ be a nonlinear Markov process. 
	For $\zeta \in \Pscr_0, 0\leq s \leq r\leq t$ and $ x \in \R^d$, define $p^{s,\zeta}_{r,t}(x,dz) \in \Pscr$ by 
	\begin{equation}\label{eq:rcp-marginals}
		p^{s,\zeta}_{r,t}(x,dz) := p_{(s,\zeta),(r,x)}\circ (\pi^r_t)^{-1}(dz) \big(= \mathbb{P}_{r,\mu^{s,\zeta}_r}[\,\,\cdot\,\,|\pi^r_r = x] \circ (\pi^r_t)^{-1}\big),
	\end{equation}
	which is uniquely determined  for $\mu^{s,\zeta}_r$-a.e. $x\in \R^d$.
	Then for $n \in \N_0$, $f \in \mathcal{B}_b((\R^d)^{n+1})$ and $s \leq t_0< \dots <t_n$:
	\begin{align}\label{eq:Markov-fdd}
		\E_{s,\zeta}&[f(\pi^s_{t_0},\dots,\pi^s_{t_n})] \\&=\notag \int_{\mathbb{R}^d}\bigg(\dots \int_{\mathbb{R}^d}\bigg(\int_{\mathbb{R}^d} f(x_0,\dots,x_n)\, p^{s,\zeta}_{t_{n-1},t_n}(x_{n-1},dx_n)\bigg) p^{s,\zeta}_{t_{n-2},t_{n-1}}(x_{n-2},dx_{n-1})\dots\bigg)\mu^{s,\zeta}_{t_0}(dx_0).
	\end{align}
\end{prop}
\begin{proof}
	Let $(s,\zeta) \in \mathbb{R}_+ \times \Pscr_0$. By a monotone class argument, it suffices to consider $f = f_0 \otimes \dots \otimes f_n$ with $f_i \in \mathcal{B}_b(\R^d)$, $n \in \N_0$, and to proceed by induction over $n$. For $n=0$,
	\begin{equation*}
		\mathbb{E}_{s,\zeta}[f_0(\pi^s_{t_0})] = \int_{\mathbb{R}^d} f_0(x_0)\, \mu^{s,\zeta}_{t_0}(dx_0)
	\end{equation*}
	holds by definition of $\mu^{s,\zeta}_{t_0}$. For $f_i \in \Bscr_b(\R^d), 0\leq i \leq n+1$, and $s \leq t_0 < \dots < t_{n+1}$, we have
	\begin{align*}
		&	\E_{s,\zeta}[\Pi_{i=0}^{n+1}f_i(\pi^s_{t_i})] \\&= \E_{s,\zeta}\big[\Pi_{i=0}^{n}f_i(\pi^s_{t_i})\mathbb{E}_{s,\zeta}[f_{n+1}(\pi^s_{t_{n+1}})|\Fscr_{s,t_n}]\big] =  \E_{s,\zeta}\big[\Pi_{i=0}^{n}f_i(\pi^s_{t_i})p_{(s,\zeta),(t_n,\pi^s_{t_{n}})}(f_{n+1}(\pi^{t_n}_{t_{n+1}}))\big] \\&=
		\int_{\R^d}\bigg(\dots\int_{\R^d} \Pi_{i=0}^{n}f_i(x_i)\bigg(	\int_{\R^d}  f_{n+1}(x_{n+1}) \,p^{s,\zeta}_{t_n,t_{n+1}}(x_n,dx_{n+1})\bigg)\, p^{s,\zeta}_{t_{n-1},t_n}(x_{n-1},dx_n)\dots\bigg)\mu^{s,\zeta}_{t_0}(dx_0),
	\end{align*}
	where we used the nonlinear Markov property and Remark \ref{rem:Markov-prop-extends-to-fct} (ii) for the second equality, and the induction assumption for the third equality.
\end{proof}

\begin{rem}\label{rem:nonlinear-feature}
	\begin{enumerate}
		\item [(i)] Hence, similarly as in the linear case, the path measure $\mathbb{P}_{s,\zeta}$ can be reconstructed from one-dimensional time marginals $p^{s,\zeta}_{r,t}(x,\cdot)$, $t \geq r \geq s \geq 0, x \in \R^d$. However, the nonlinear nature becomes evident via the fact that even in the case $\Pscr_0 = \Pscr$ it is usually \textup{not} true that $p^{s,\zeta}_{r,t}(x,\cdot) = \mathbb{P}_{r,\delta_x}\circ (\pi^r_t)^{-1}$, see Remark \ref{rem:final-rem}.
		\item[(ii)] By Proposition \ref{prop:Markov-distr-from-marginals}, the nonlinear Markov property extends to
		$$\mathbb{E}_{s,\zeta}\big[h(\pi^s_{t_1},\dots \pi^s_{t_n})|\mathcal{F}_{s,r}\big](\cdot) = p_{(s,\zeta),(r,\pi^s_r(\cdot))}\big(h(\pi^r_{t_1},\dots,\pi^r_{t_n})\big)\quad \mathbb{P}_{s,\zeta}-\text{a.s.}$$
		for all $0\leq s \leq r \leq t_1< \dots < t_n$, $n\in \N$, $\zeta \in \Pscr_0$ and $h \in \Bscr_b{((\R^d)^n)}$. By another application of the monotone class theorem, this generalizes to
		$$\mathbb{E}_{s,\zeta}\big[G\circ \Pi^s_r|\mathcal{F}_{s,r}\big](\cdot) = p_{(s,\zeta),(r,\pi^s_r(\cdot))}\big(G\big)\quad \mathbb{P}_{s,\zeta}-\text{a.s.}$$
		for all $0\leq s \leq r$, $\zeta \in \Pscr_0$ and $G \in \Bscr_b(\Omega_r)$. Note that a function $\tilde{G}:\Omega_s \to \R$ is bounded and $\sigma(\pi^s_\tau, \tau \geq r)$-measurable if and only if $\tilde{G} = G\circ \Pi^s_r$ with $G \in \Bscr_b(\Omega_r)$.
		
	\end{enumerate}
	
\end{rem}

\subsection{Comparison with linear Markov processes}\label{sect:comp-to-lin}
Here we comment on the connection to classical linear Markov processes. More precisely, we show that our Definition \ref{def:NL-Markov-process} is a proper extension of the linear definition. Afterwards we briefly recall how classical Markov processes typically arise from linear Kolmogorov operators. 

Let $(\mathbb{P}_{s,x})_{(s,x) \in \R_+\times \R^d}$ be a linear normal time-inhomogeneous \textit{Markov process}, i.e. for each $(s,x)$, $\mathbb{P}_{s,x}$ is a probability measure on $\Bscr(\Omega_s)$ with $\mathbb{P}_{s,x}\circ (\pi^s_s)^{-1} = \delta_x$, such that 
\begin{enumerate}
	\item[(i)] $x \mapsto \mathbb{P}_{s,x}(B)$ is measurable for all $s \geq 0$ and $B \in \Bscr(\Omega_s)$,
	\item[(ii)] the linear Markov property holds, i.e. for all $s\leq r \leq t$ 
	\begin{equation*}
		\mathbb{P}_{s,x}(\pi^s_t \in A|\Fscr_{s,r})(\cdot) = \mathbb{P}_{r,\pi^s_r(\cdot)}(\pi^r_t \in A)\quad \mathbb{P}_{s,x}-\text{a.s.} \text{ for all }A \in \Bscr(\R^d).
	\end{equation*}
\end{enumerate}
\begin{rem}\label{rem:linear-case-convex}
	It is natural to set $\mathbb{P}_{s,\zeta} := \int_{\mathbb{R}^d} \mathbb{P}_{s,y}\,\zeta(dy)$ for any $s\geq 0$ and non-Dirac $\zeta \in \Pscr$. Then $\mathbb{P}_{s,\zeta}\circ (\pi^s_s)^{-1} = \zeta$, and we again define $\mu^{s,\zeta}_t$, $t \geq s$, as in Definition \ref{def:NL-Markov-process} (i). Then clearly $\mu^{s,\zeta}_t = \int_{\R^d}\mu^{s,y}_t\,\zeta(dy)$, and for any $0\leq s \leq t$, the map $\Pscr\ni \zeta \mapsto \mu^{s,\zeta}_t$ is linear on its domain.
\end{rem}

\begin{prop}\label{prop:linMarkov-is-nonlinMarkov}
	$(\mathbb{P}_{s,\zeta})_{(s,\zeta)\in \mathbb{R}_+\times \Pscr}$ is a nonlinear Markov process with $\Pscr_0 = \Pscr$ in the sense of Definition \ref{def:NL-Markov-process}.
\end{prop}
\begin{proof}
	
	We have $\mathbb{P}_{r,\mu^{s,\zeta}_r} = \int_{\mathbb{R}^d} \mathbb{P}_{r,y}\,\mu^{s,\zeta}_r(dy)$ and $\mathbb{P}_{r,y}$ is concentrated on $\{\pi^r_r = y\}$. Hence $\mathbb{P}_{r,y}, y \in \R^d$, is  a regular conditional probability kernel of $\mathbb{P}_{r,\mu^{s,\zeta}_r}[\,\cdot \, | \pi^r_r=y]$, $y \in \R^d$, and thus \eqref{Markov-prop} holds with $p_{(s,\zeta),(r,\pi^s_r(\cdot))} = \mathbb{P}_{r,\pi^s_r(\cdot)}$.
\end{proof}

\begin{rem}
	The transition probabilities $\mu^{s,x}_t := \mu^{s,\delta_x}_t = \mathbb{P}_{s,x}\circ (\pi^s_t)^{-1}$ of a linear Markov process satisfy the \textup{Chapman-Kolmogorov equations}
	\begin{equation}\tag{CK}\label{CK}
		\mu^{s,x}_{t}(A) =\int_{\mathbb{R}^d} \mu^{r,y}_t(A)\,\mu^{s,x}_r(dy),\quad \forall\, 0\leq s \leq r \leq t, x\in \R^d \text{ and }A \in \Bscr(\R^d).
	\end{equation}
	They also satisfy the flow property, since
	\begin{align*}
		\mu^{s,\zeta}_t(A) &= \mathbb{E}_{s,\zeta}\big[\mathbb{P}_{s,\zeta}(\pi^s_t\in A|\Fscr_{s,r})\big]\\& = \mathbb{E}_{s,\zeta}\big[\mathbb{P}_{r,\pi^s_r}(\pi^r_t \in A)\big] = \int_{\R^d} \mathbb{P}_{r,y}(\pi^r_t \in A)\,\mu^{s,\zeta}_r(dy) = \mathbb{P}_{r,\mu^{s,\zeta}_r}(\pi^r_t \in A)= \mu^{r,\mu^{s,\zeta}_r}_t(A)
	\end{align*}
	(of course this also follows directly from Proposition \eqref{prop:linMarkov-is-nonlinMarkov} and Remark \ref{rem:Markov-prop-extends-to-fct} (iii)). To us it seems that the Chapman-Kolmogorov equations do not have a natural counterpart in the nonlinear case, since even in the rather special case $\Pscr_0 = \Pscr$,  \eqref{CK} will usually not be satisfied. Indeed, \eqref{CK} is not satisfied  when the curves $(\mu^{s,x}_t)_{t \geq s}$ are solutions to a nonlinear FPKE, which is the case for all our examples in Section \ref{sect:Appl}.
	From our point of view, with regard to nonlinear Markov processes, the flow property is the key property on the level of the one-dimensional time marginals, see also our main result, Theorem \ref{theorem1}.
\end{rem}

\subsubsection{Linear Markov processes from linear Kolmogorov operators} Before we continue with the nonlinear case, let us recall how linear Markov processes typically arise from linear Kolmogorov operators $L$ of type
\begin{equation*}
	L_t\varphi(x) = a_{ij}(t,x)\partial^2_{ij}\varphi(x) + b_i(t,x)\partial_i \varphi(x),\quad \varphi \in C^2(\R^d),
\end{equation*}
where $a_{ij}$ and $b_i$, $1 \leq i,j \leq d$, are Borel coefficients on $\R_+\times \R^d$ such that $a := (a_{ij})_{i,j \leq d}$ is symmetric and non-negative definite. The stochastic differential equation (SDE) associated with $L$ is
\begin{equation*}
	dX_t = b(t,X_t)dt+\sigma(t,X_t)dB_t,\quad \quad\quad  t\geq s,
\end{equation*}
where $\sigma = (\sigma_{ij})_{1\leq i,j\leq d}$ is a square root of $a$ so that $1/2\sigma \sigma^T = a$, and $B$ is a Brownian motion in $\R^d$. Probabilistic weak solutions to this SDE are equivalent to solutions to the martingale problem for $L$ (see \cite[Prop.4.11]{KS91}, and also \cite{StroockVaradh2007}). By Itô's formula and the Ambrosio-Figalli-Trevisan superposition principle (see for instance \cite[Thm.2.5]{Trevisan16}), $X^{s,x}$ is a weak solution to the SDE with $X^{s,x}_s = x$, if and only if its one-dimensional time marginal curve $t\mapsto \mu^{s,x}_t$ solves the Cauchy problem of the FPKE
\begin{equation}\label{FP-eq}
	\partial_t \mu_t = L^*_t \mu_t,\quad t\geq s, \quad \mu_s = \delta_x,
\end{equation}
in the distributional sense ($L^*$ denotes the formal dual operator of $L$) and satisfies a mild integrability condition with respect to $a$ and $b$. By the aforementioned superposition principle and by \cite[Prop.4.11]{KS91}, if either the Cauchy problem for the martingale problem, the SDE or the FPKE is well-posed for every initial datum $(s,x)$, then all three problems are well-posed. In this case, the path laws $\mathbb{P}_{s,x}$ of the unique weak solutions to the SDE (which are exactly the martingale problem solutions) constitute a linear Markov process.

The question whether a Markov process exists in ill-posed regimes is more delicate, but positive answers are known: If no uniqueness is known at all, at least for continuous bounded coefficients one can select one particular martingale solution $\mathbb{P}_{s,x}$ for each initial datum such that the selected family constitutes a Markov process, cf. \cite[Ch.12]{StroockVaradh2007}.
If partial well-posedness holds in the sense that for suitably rich subclasses of initial conditions $\mathcal{R}_s \subseteq \Pscr$, the Cauchy problem \eqref{FP-eq} from initial time $s$ has a unique solution in a suitable subclass $\mathcal{R}_{[s,\infty)}$, then for all $\zeta \in \Rscr_s$ the corresponding martingale problem has a unique solution $\mathbb{P}_{s,\zeta}$ with one-dimensional time marginals in $\mathcal{R}_{[s,\infty)}$, and the path measures $\mathbb{P}_{s,\zeta}$ induce a Markov process via disintegration, see \cite[Prop.2.8]{Trevisan16}.

\section{Construction of nonlinear Markov processes: Main result}\label{sect:main-result}
In this section we present our main result on the construction of nonlinear Markov processes from prescribed one-dimensional time marginals given as solution curves to nonlinear FPKEs. First, let us introduce the analytic and stochastic equations associated with a nonlinear Kolmogorov operator.

\subsection{Equations and solutions}\label{subsect:eq-sol}
Let $a_{ij}, b_i: \mathbb{R}_+\times \Pscr\times \R^d \to \R$, $1 \leq i,j \leq d$, be Borel maps (with respect to the weak topology on $\Pscr$), set $a := (a_{ij})_{ i,j \leq d}, b:= (b_i)_{i \leq d}$, and consider for $(t,\mu)\in \mathbb{R}_+\times \Pscr$ the nonlinear Kolmogorov operator
\begin{equation}\label{op-L}
	L_{t,\mu}\varphi (x) = a_{ij}(t,\mu,x)\partial_{ij}^2\varphi(x)+b_i(t,\mu,x)\partial_i \varphi(x),\quad \varphi \in C^2(\R^d).
\end{equation}
We assume $a = \frac{1}{2}\sigma \sigma ^T$ for some Borel map $\sigma: \mathbb{R}_+\times {\Pscr}\times \R^d \to \R^{d\times d}$, with coefficients $\sigma_{ij}$, and we denote by $B$ an $\R^d$-standard Brownian motion. For a random variable $Z$, let $\mathcal{L}_Z$ denote its distribution. In particular, for a process $X$ with paths in $\Omega_s$, $\mathcal{L}_X$ denotes its path law on $\Omega_s$, and $\mathcal{L}_{X_t}$, $t \geq s$, its one-dimensional time marginals on $\R^d$.

As in the linear case, three nonlinear problems are associated with $L$ and an initial datum $(s,\zeta)\in \mathbb{R}_+\times \Pscr$: On the level of path measures the distribution-dependent stochastic differential equation (also \textit{McKean-Vlasov equation})
\begin{equation}\tag{DDSDE}\label{DDSDE}
	dX_t = b(t,\mathcal{L}_{X_t},X_t)dt+\sigma(t,\mathcal{L}_{X_t},X_t)dB_t,\quad t \geq s,\quad \mathcal{L}_{X_s} = \zeta,
\end{equation}
(the same equation as in \eqref{intro:DDSDE})
and the nonlinear martingale problem, which consists of finding $P \in \mathcal{P}(\Omega_s)$ such that
\begin{equation}\label{MGP}\tag{MGP}
	P\circ (\pi^s_s)^{-1} = \zeta \,\text{  and  }\,	\varphi(\pi^s_t)-\int_s^t L_{r,\mathcal{L}_{\pi^s_r}}\varphi(\pi^s_r)\, dr , \,\,t \geq s,\text{ is a martingale }\forall \varphi \in C^2_c(\R^d)
\end{equation}
with respect to $(\mathcal{F}_{s,t})_{t \geq s}$. On state space level, one has the nonlinear FPKE
\begin{equation}\tag{FPE}\label{NLFPE}
	\partial_t \mu_t = L^*_{t,\mu_t}\mu_t,\quad t\geq s,\quad  \mu_s = \zeta,
\end{equation}
where $L_{t,\mu}^*$ denotes the formal dual operator of $L_{t,\mu}$ (this is the same equation as \eqref{intro:NLFPE}).
For an initial datum $(s,\eta) \in \R_+\times \Pscr$, from \eqref{NLFPE} one obtains a \textit{linear} FPKE by fixing a weakly continuous curve $[s,\infty) \ni t \mapsto \mu_t \in \mathcal{P}$ in $a$ and $b$:
\begin{equation}\tag{$\ell$FPE}\label{LFPE}
	\partial_t \nu_t = L^*_{t,\mu_t}\nu_t,\quad t\geq s,\quad \nu_s = \eta.
\end{equation}
Analogously, one obtains a (non-distribution dependent) SDE from \eqref{DDSDE}:
\begin{equation}\tag{$\ell$DDSDE}\label{LDDSDE}
	dX_t = b(t,\mu_t,X_t)dt+\sigma(t,\mu_t,X_t)dB_t,\quad t\geq s, \quad \mathcal{L}_{X_s} = \eta,
\end{equation}
without any a priori relation between $X_t$ and $\mu_t$,
and a linear martingale problem, namely to find $P \in \Pscr(\Omega_s)$ such that
\begin{equation}\tag{$\ell$MGP}\label{linMGP}
	P\circ (\pi^s_s)^{-1} = \eta \,\text{  and  }\,	\varphi(\pi^s_t)-\int_s^t L_{r,\mu_r}\varphi(\pi^s_r)\, dr,\,\, t \geq s, \text{ is a martingale }\forall \varphi \in C^2_c(\R^d)
\end{equation}
with respect to $(\mathcal{F}_{s,t})_{t \geq s}$.

\begin{dfn}\label{def:all-single-eq.}
	Let $a_{ij},b_i,\sigma_{ij}: \mathbb{R}_+\times \Pscr\times \mathbb{R}^d \to \mathbb{R}$ be as above, $s\geq 0$ and $\zeta, \eta \in \Pscr$.
	\begin{enumerate}
		\item [(a)] \textbf{Solutions to \eqref{DDSDE}, \eqref{MGP}, \eqref{NLFPE} and linear counterparts.}
		\begin{enumerate}
			\item[(i)] A stochastic process $(X_t)_{t \geq s}$ on a filtered probability space $(\Omega, \Fscr,(\Fscr_t)_{t \geq s},\mathbb{P},B)$ with an $\R^d$-valued $(\mathcal{F}_t)$-adapted Brownian motion $B$ is a \textit{(probabilistic weak) solution to \eqref{DDSDE}} from $(s,\zeta)$, if it is $(\mathcal{F}_t)$-adapted, $\mathcal{L}_{X_s} = \zeta$,
			\begin{equation}\label{int-DDSDE-sol}
				\mathbb{E}\bigg[	\int_s^T|b(t,\Lscr_{X_t},X_t)|+|\sigma(t,\Lscr_{X_t},X_t)|^2  dt\bigg] < \infty,\quad \forall T>s,
			\end{equation}
			and $\mathbb{P}$-a.s.
			\begin{equation}\label{eq-DDSDE-sol}
				X_t-X_s = \int_s^t b(r,\mathcal{L}_{X_r},X_r)dr+\int_s^t \sigma(r,\mathcal{L}_{X_r},X_r)dB_r, \quad \forall t \geq s.
			\end{equation}
			\item[(ii)] A process $X$ as in (i) is a solution to \eqref{LDDSDE} from $(s,\eta)$, if $\mathcal{L}_{X_s} = \eta$ and \eqref{int-DDSDE-sol} and \eqref{eq-DDSDE-sol} hold with $\mu_t$ instead of $\mathcal{L}_{X_t}$, where $[s,\infty) \ni t \mapsto \mu_t \in \Pscr$ is weakly continuous.
			\item[(iii)] $P \in \mathcal{P}(\Omega_s)$ is a \textit{solution to \eqref{MGP}} from $(s,\zeta)$, if
			\begin{equation}\label{global-int-NLMGP}
				\mathbb{E}_P\bigg[\int_s^T |b(t,\mathcal{L}_{\pi^s_t},\pi^s_t)|+|a(t,\mathcal{L}_{\pi^s_t},\pi^s_t)|\,dt\bigg] <\infty, \quad \forall T>s,
			\end{equation}
			and $P$ solves \eqref{MGP}.
			\item[(iv)] Similarly, $P$ is a solution to \eqref{linMGP} from $(s,\eta)$, if \eqref{global-int-NLMGP} holds, with $\mu_t$ replacing $\mathcal{L}_{\pi^s_t}$, and $P$ solves \eqref{linMGP}.
			\item [(v)] A weakly continuous curve $\mu:[s,\infty)\ni t \mapsto \mu_t$ in $\mathcal{P}$ is a \textit{solution to \eqref{NLFPE}} from $(s,\zeta)$, if
			\begin{equation*}
				\int_s^T \int_{\mathbb{R}^d}|a(t,\mu_t,x)|+|b(t,\mu_t,x)|\,\mu_t(dx)dt < \infty, \quad \forall T> s,
			\end{equation*}
			and for all $\varphi \in C^2_c$
			\begin{equation*}
				\int_{\mathbb{R}^d}\varphi (x)\,\mu_t(dx) - \int_{\mathbb{R}^d} \varphi (x)\,\zeta(dx)= \int_s^t\int_{\mathbb{R}^d} L_{\tau,\mu_\tau}\varphi(x) \,\mu_\tau(dx) d\tau,\quad \forall t\geq s.
			\end{equation*}
			
			\item[(vi)] For $\mathcal{P}_0 \subseteq \Pscr$, a family $\{\mu^{s,\zeta}\}_{(s,\zeta)\in \R_+\times \Pscr_0}$, $\mu^{s,\zeta} = (\mu^{s,\zeta}_t)_{t \geq s}$, of solutions to \eqref{NLFPE} with initial datum $\mu^{s,\zeta}_s = \zeta$ is a \textit{solution flow}, if it satisfies the flow property, i.e. if $\mu^{s,\zeta}_t \in \Pscr_0$ for all $s\leq t, \zeta \in \Pscr_0$, and
			\begin{equation*}
				\mu^{s,\zeta}_t = \mu^{r,\mu^{s,\zeta}_r}_t,\quad \forall\,0\leq  s \leq r \leq t, \zeta \in \Pscr_0.
			\end{equation*}
			\item[(vii)] A weakly continuous curve $\nu: [s,\infty)\ni t \mapsto \nu_t$ in $\Pscr$ is a \textit{solution to \eqref{LFPE}} from $(s,\eta)$, if 
			\begin{equation*}
				\int_s^T \int_{\mathbb{R}^d}|a(t,\mu_t,x)|+|b(t,\mu_t,x)|\,\nu_t(dx)dt < \infty, \quad \forall T> s,
			\end{equation*}
			and for all $\varphi \in C^2_c(\R^d)$
			\begin{equation*}
				\int_{\mathbb{R}^d}\varphi (x)\,\nu_t(dx) - \int_{\mathbb{R}^d} \varphi \,\eta (dx)= \int_s^t\int_{\mathbb{R}^d} L_{\tau,\mu_\tau}\varphi(x) \,\nu_\tau(dx) d\tau ,\quad \forall t\geq s.
			\end{equation*}
		\end{enumerate}
		\item[(b)] \textbf{Uniqueness of solutions.} 
		\begin{enumerate}
			\item[(i)] Solutions to \eqref{DDSDE} from $(s,\zeta)$ are weakly unique, if for any two solutions $X^1,X^2$, $\mathcal{L}_{X^i_s} = \zeta$, $i \in \{1,2\}$, implies $\mathcal{L}_{X^1} = \mathcal{L}_{X^2}$.
			\item[(ii)] Solutions to \eqref{MGP} from $(s,\zeta)$ are unique, if for any two solutions $P^1,P^2$, $P^i\circ (\pi^s_s)^{-1} = \zeta$, $i \in \{1,2\}$, implies $P^1=P^2$ on $\Bscr(\Omega_s)$.
			\item[(iii)] Solutions to \eqref{NLFPE} from $(s,\zeta)$ are unique, if for any two solutions $\mu^1,\mu^2$, $\mu^i_s = \zeta$, $i \in \{1,2\}$, implies $\mu_t^1 = \mu_t^2$ for all $t \geq s$.
			\item[(iv)] 
			For a fixed weakly continuous curve $[s,\infty) \ni t \mapsto \mu_t  \in \Pscr$, uniqueness of solutions to \eqref{LDDSDE}, \eqref{linMGP} and \eqref{LFPE} from an initial condition $(s,\eta) \in \R_+\times \Pscr$ is defined analogously.
		\end{enumerate}
		For each of these problems, \textit{uniqueness in a subclass of solutions} means any two solutions from a specified subclass with identical initial datum coincide.
	\end{enumerate}
\end{dfn}
\begin{rem}\label{rem:known-relations-of-problems}
	The following relations between solutions to the above problems are well-known. 
	\begin{enumerate}
		\item 	[(i)] The path law $P$ of a weak solution $X$ to \eqref{DDSDE}, resp. \eqref{LDDSDE}, solves \eqref{MGP}, resp. \eqref{linMGP}, and vice versa, for a solution $P$ to \eqref{MGP}, resp. \eqref{linMGP}, there exists a weak solution $X$ to \eqref{DDSDE}, resp. \eqref{LDDSDE} such that $P = \mathcal{L}_X$. These equivalences are in one-to-one correspondence, i.e. existence and uniqueness of solutions for the [linear] stochastic equation and the [linear] martingale problem are equivalent, see for instance \cite[Prop.4.11]{KS91}.
		\item[(ii)] By Itô's formula, a solution $P$ to \eqref{MGP}, resp. \eqref{linMGP}, induces a solution to the nonlinear, resp. linear FPKE via its curve of one-dimensional time marginals. Conversely, by the Ambrosio-Figalli-Trevisan-superposition principle, for any solution $(\nu_t)_{t \geq s}$ to \eqref{LFPE} there is a solution to \eqref{linMGP} with marginals $(\nu_t)_{t \geq s}$, cf. \cite{Ambrosio08,Figalli09,Trevisan16,BRS19-SPpr}. An analogous version of this result for nonlinear equations is due to Barbu and the second-named author of this paper, see \cite{BR18,BR18_2}. Hence uniqueness for the [linear] martingale problem yields uniqueness for the [linear] FPKE. Conversely, uniqueness of \eqref{LFPE} for sufficiently many initial data implies uniqueness for \eqref{linMGP}. To obtain uniqueness of solutions to \eqref{MGP}, one needs not only uniqueness for \eqref{NLFPE}, but also for the associated linearized FPKE, see \cite{BR22,RRW20}.
	\end{enumerate}
\end{rem}
\begin{rem}\label{rem:new}
	If $\{\mathbb{P}_{s,\zeta}\}_{(s,\zeta)\in \R_+\times \Pscr_0}$ is a nonlinear Markov process, consisting of solution laws to an equation of type \eqref{DDSDE}, its one-dimensional time marginal curves $(\mu^{s,\zeta}_t)_{t\geq s}$, $\mu^{s,\zeta}_t:=\mathbb{P}_{s,\zeta}\circ (\pi^s_t)^{-1}$, solve \eqref{NLFPE}, and $(p^{s,\zeta}_{r,t}(x,dz))_{t\geq r}$ as in \eqref{eq:rcp-marginals} solves \eqref{LFPE} with $\mu^{s,\zeta}_t$ replacing $\mu_t$, and with initial datum $(r,\delta_x)$  for $\mu^{s,\zeta}_r$-a.e. $x$. This follows from \cite[Prop.2.8]{Trevisan16} and part (ii) of the previous remark. Hence, if solutions to \eqref{LFPE} with $\eta = \delta_x$ and $\mu^{s,\zeta}_t$ replacing $\mu_t$ are unique for every $(s,\zeta,x) \in\R_+\times \Pscr_0 \times \R^d$, then the kernels from \eqref{eq:rcp-marginals} are the unique solutions to these linear equations. Then $p^{s,\zeta}_{r,t}$, $s\leq r \leq t$, are the transition kernels of a linear time-inhomogeneous Markov process $\{P^{s,\zeta}_{r,x}\}_{(r,x)\in [s,\infty)\times \R^d}$. Their relation to the nonlinear Markov process is given by
	\begin{equation*}
		\mathbb{P}_{r,\mu^{s,\zeta}_r} = \int_{\R^d} P^{s,\zeta}_{r,x}\,d\mu^{s,\zeta}_r(x),\quad \forall 0\leq s\leq r, \zeta \in \Pscr_0
	\end{equation*}
(i.e. the RHS is the convex mixture of the path laws of the corresponding linear Markov process).
In this case, \eqref{eq:Markov-fdd} shows that the finite-dimensional marginals of $\mathbb{P}_{s,\zeta}$ (and hence its path law) are uniqiuely determined by the transition kernels of a linear Markov process, which depends, however, on $(s,\zeta)$.
\end{rem}

\subsection{Main result}\label{subsect:main-result}
We want to construct nonlinear Markov processes with prescribed one-dimensional time marginals, which are given as solutions to nonlinear FPKEs, see Theorem \ref{theorem1}, which is the main result of this paper. Before its formulation and proof, we present another main result of this section, Theorem \ref{aux-prop-ex+uniqu}, which gives a new condition for uniqueness of solutions to \eqref{DDSDE} with prescribed one-dimensional time marginals $(\mu^{s,\zeta}_t)_{t \geq s\geq 0, \zeta \in \Pscr_0}$ in terms of the extremality of each $(\mu^{s,\zeta}_t)_{t \geq s}$ in the set of solutions to \eqref{LFPE}, with $\mu^{s,\zeta}_t$ replacing $\mu_t$.

To this end, we introduce the following notation. For $(s,\zeta)\in \mathbb{R}_+\times \Pscr$, we denote the space of all solutions to \eqref{NLFPE} from $(s,\zeta)$ by $M^{s,\zeta}$ and, for a weakly continuous curve $\eta: [s,\infty)\in t \mapsto \eta_t \in \Pscr$, the space of all solutions to \eqref{LFPE} from $(s,\zeta)$, with $\eta_t$ replacing $\mu_t$, by $M^{s,\zeta}_\eta$. Here, as throughout the paper, solutions to \eqref{NLFPE}, resp. \eqref{LFPE}, are understood as in Definition \ref{def:all-single-eq.} (v), resp. (vii), i.e. in particular all elements of $M^{s,\zeta}$ and $M^{s,\zeta}_\eta$ are weakly continuous curves of probability measures.

Clearly $M^{s,\zeta}_\eta$ is a convex set for all $(s,\zeta) \in \R_+\times \Pscr$ and $\eta$ as above, and we recall that $\mu$ is an \emph{extremal point} of $M^{s,\zeta}_\eta$, if $\mu \in M^{s,\zeta}_\eta$ and $\mu = \alpha \mu^1 + (1-\alpha)\mu^2$ for $\alpha \in (0,1)$ and $\mu^1,\mu^2 \in M^{s,\zeta}_\eta$ implies $\mu^1=\mu^2$. The set of all extremal points of $M^{s,\zeta}_\eta$ is denoted by $M^{s,\zeta}_{\eta,\text{ex}}$.

\begin{theorem}\label{aux-prop-ex+uniqu}
	Let $\Pscr_0 \subseteq\Pscr$ and assume that $\{\mu^{s,\zeta}\}_{(s,\zeta)\in \R_+\times \Pscr_0}$ is a solution flow to \eqref{NLFPE} such that $\mu^{s,\zeta} \in M^{s,\zeta}_{\mu^{s,\zeta}, \text{ex}}$ for each $(s,\zeta) \in \R_+\times \Pscr_0$. Then for every $(s,\zeta) \in \R_+\times \Pscr_0$, there is a unique weak solution $X^{s,\zeta}$
	to the corresponding equation \eqref{DDSDE} with initial condition $(s,\zeta)$ and one-dimensional time marginals equal to $(\mu^{s,\zeta}_t)_{t \geq s}$.
\end{theorem}

Regarding the proof of Theorem \ref{aux-prop-ex+uniqu}, we first provide a convenient characterization of the extremality assumption in terms of an as mild as possible restricted uniqueness property of the linearized equations \eqref{LFPE} associated with \eqref{NLFPE}.
For a $\Pscr$-valued curve $\mu=(\mu_t)_{t \geq s}$ and $C>0$ set
$$\Ascr_{s,\leq}(\mu,C):= \big\{(\eta_t)_{t \geq s}\in C([s,\infty),\Pscr): \eta_t \leq C \mu_t, \,\forall  t\geq s \big\},\,\,\,\, \Ascr_{s,\leq}(\mu) := \bigcup_{C>0} \Ascr_{s,\leq}(\mu,C).$$

In the following lemma we consider a linear FPKE and denote the space of its solutions from $(s,\zeta) \in \R_+\times \Pscr$ by $M^{s,\zeta}$, which is no abuse of notation, since linear FPKES are special cases of nonlinear FPKES of type \eqref{NLFPE} with coefficients independent of their measure variable. Moreover, in this linear case we denote the set of extremal points of $M^{s,\zeta}$ by $M^{s,\zeta}_{\text{ex}}$. For a set $A$, $\#A$ denotes its cardinality.

\begin{lem}\label{lem:Michael-observation}
	Consider a linear FPKE, i.e. \eqref{NLFPE} with coefficients independent of their measure variable. Let $(s,\zeta) \in \R_+\times \Pscr$ and $\mu = (\mu_t)_{t \geq s} \in M^{s,\zeta}$. Then 
	$$\#(M^{s,\zeta} \cap \Ascr_{s,\leq}(\mu)) = 1 \iff \mu \in M^{s,\zeta}_{\text{ex}}.$$
\end{lem}
\begin{proof}
	Fix $(s,\zeta) \in \R_+\times \Pscr$.
	Clearly, $\mu \in M^{s,\zeta} \cap \Ascr_{s,\leq}(\mu)$. First, suppose $\mu \notin M^{s,\zeta}_{\text{ex}}$, i.e. there are $\mu^i, i \in \{1,2\}$, in $M^{s,\zeta}$ and $\alpha \in (0,1)$ such that 
	\begin{equation}\label{aux-proof-lemma3.4.}
		\mu_t = \alpha \mu_t^1 + (1-\alpha)\mu_t^2, \quad t \geq s,
	\end{equation}
	and $\mu^1 \neq \mu^2$. Then \eqref{aux-proof-lemma3.4.} implies $\mu^i \in M^{s,\zeta} \cap \Ascr_{s,\leq}(\mu)$, $i \in \{1,2\}$, and hence
	$\#(M^{s,\zeta} \cap \Ascr_{s,\leq}(\mu)) \geq 2$.
	
	Now assume $\mu \in M^{s,\zeta}_{\text{ex}}$ and let $\nu \in M^{s,\zeta}\cap \Ascr_{s,\leq}(\mu)$. Then for every $t \geq s$ there is $\vrho_t : \R^d \to \R_+$, $\Bscr(\R^d)$-measurable, such that $\nu_t = \vrho_t \, \mu_t$, and $\vrho_t \leq C$ for all $t \geq s$ for some $C\in (1,\infty)$. Furthermore, for $t \geq s$,
	\begin{align*}
		\mu_t = \frac 1 C \vrho_t \,\mu_t + (1-\frac 1 C \vrho_t) \, \mu_t = \frac 1 C \nu_t + (1-\frac 1 C)\lambda_t,
	\end{align*}
	where $\lambda_t := \frac{1-\frac 1 C \vrho_t}{1- \frac 1 C}\,\mu_t$. Clearly $\lambda_t(\R^d)=1$ and thus $(\lambda_t)_{t \geq s} \in M^{s,\zeta}$. Since $\mu^{s,\zeta} \in M^{s,\zeta}_{\text{ex}}$, it follows that $\mu_t = \nu_t$.
\end{proof}

\begin{rem}
	By Lemma \ref{lem:Michael-observation}, one obtains an equivalent formulation of Theorem \ref{aux-prop-ex+uniqu}, which states that the restricted uniqueness in the subclasses $\Ascr_{s,\leq}(\mu^{s,\zeta})$ of solutions to \eqref{LFPE}, with $\mu^{s,\zeta}_t$ replacing $\mu_t$, for all $(s,\zeta) \in \R_+\times \Pscr_0$ implies the restricted uniqueness in the subclasses of solutions with one-dimensional time marginals $(\mu^{s,\zeta}_t)_{t \geq s}$ of solutions to \eqref{DDSDE} for all $(s,\zeta) \in \R_+\times \Pscr_0$.
\end{rem}

As a further preparation for the proofs of Theorem \ref{aux-prop-ex+uniqu} and Theorem \ref{theorem1}, we need Part (ii) of  Lemma \ref{lem:uniqueness-extends} below. Its Part (i) is not used in this paper, but is of independent interest. We postpone the proof of Lemma \ref{lem:uniqueness-extends} until the end of this section.

\begin{lem}\label{lem:uniqueness-extends}
	Consider a linear FPKE, i.e. \eqref{NLFPE} with coefficients independent of their measure variable, and let $(s_0,\zeta_0) \in \R_+\times \Pscr$. Then the following holds.
	\begin{enumerate}
		\item[(i)] If solutions (in the sense of Definition \ref{def:all-single-eq.} (v) with $a,b$ independent of $\mu \in \Pscr$) are unique from $(s_0,\zeta_0)$, then solutions are also unique from any $(s_0,\eta)$ such that $\eta \in \Pscr$ and $\eta \sim \zeta_0$.
		\item[(ii)] If $(\nu^{s_0,\zeta_0}_t)_{t \geq s}$ is the unique solution in $\Ascr_{s_0,\leq}(\nu^{s_0,\zeta_0})$
		from $(s_0,\zeta_0)$, then in this class solutions are also unique from any $(s_0,g \,\zeta_0)$ with $g \in \mathcal{B}^+_b(\R^d)$, $\int_{\mathbb{R}^d} g(x) \,\zeta_0(dx) =1,$ and $\delta \leq g$ for some $\delta >0$.
	\end{enumerate}
\end{lem}

We proceed to the proof of Theorem \ref{aux-prop-ex+uniqu}.
\begin{proof}[Proof of Theorem \ref{aux-prop-ex+uniqu}]
	The existence of a weak solution $X^{s,\zeta}$ to \eqref{DDSDE} follows from \cite[Sect.2]{BR18_2}. Concerning uniqueness, note that by the assumption and Lemma \ref{lem:Michael-observation}, for each $0\leq s \leq r$, $\zeta \in \Pscr_0$, \eqref{LFPE}, with $\mu^{s,\zeta}_t = \mu^{r,\mu^{s,\zeta}_r}_t$ replacing $\mu_t$, has a unique solution from $(r, \mu^{s,\zeta}_r)$ in $\Ascr_{r,\leq}(\mu^{r,\mu^{s,\zeta}_r})$. 
	
	\textit{Claim:} Together with Lemma \ref{lem:uniqueness-extends} (ii), this implies that for any $(s,\zeta) \in \R_+\times \Pscr_0$ and $r \geq s$, even solutions to the linearized martingale problem \eqref{linMGP}, with $\mu^{r,\mu^{s,\zeta}_r}_t$ replacing $\mu_t$, with one-dimensional time marginals in $\Ascr_{r,\leq}(\mu^{r,\mu^{s,\zeta}_r})$ are unique from $(r,\mu^{s,\zeta}_r)$ (by the superposition principle, this is indeed stronger than uniqueness for the linearized FPKEs). 
	
	\textit{Proof of Claim:}
	In fact, this uniqueness for \eqref{linMGP} can be obtained by Lemma \ref{lem:uniqueness-extends} (ii) and a refinement of the proof of \cite[Lem.2.12]{Trevisan16} as follows: Fix $(s,\zeta)\in \R_+\times \Pscr_0$, $r\geq s$, and let $P^1$ and $P^2$ be two solutions to \eqref{linMGP}, with $\mu^{s,\zeta}_t = \mu^{r,\mu^{s,\zeta}_r}_t$ replacing $\mu_t$, with one-dimensional time marginals in $\Ascr_{r,\leq}(\mu^{r,\mu^{s,\zeta}_r})$ from $(r,\mu^{s,\zeta}_r)$. Then the one-dimensional time marginal curves $(P^i_t)_{t \geq r}, P^i_t := P^i\circ (\pi^r_t)^{-1}$, solve \eqref{LFPE}, with $\mu^{r,\mu^{s,\zeta}_r}_t = \mu^{s,\zeta}_t$ replacing $\mu_t$, from $(r,\mu^{s,\zeta}_r)$, and hence
	\begin{equation}\label{aux-2}
		P^i_t = \mu^{r,\mu^{s,\zeta}_r}_t = \mu^{s,\zeta}_t, \,\,\forall \, t\geq r, \quad i\in \{1,2\}.
	\end{equation}
	For $n\in \N$, let 
	$$\Hscr^{(n)}_r:= \{\Pi_{i=1}^nh_i(\pi^r_{t_i}) \,| \,h_i \in \Bscr^+_b(\R^d), h_i \geq c_i \text{ for some }c_i>0, r\leq t_1 < \dots < t_n\},$$
	$$\Hscr_r := \bigcup_{n \in \mathbb{N}}\Hscr_r^{(n)}$$
	and note that $\Hscr_r$ is closed under pointwise multiplication and $\sigma(\Hscr_r) = \Bscr(\Omega_r)$. Hence, by induction in $n \in \N$ and a monotone class argument, it suffices to prove
	\begin{equation}\label{aux-eq-Trevisan-proof}
		\mathbb{E}_{P^1}[H] = \mathbb{E}_{P^2}[H] \text{ for all } H \in \Hscr^{(n)}_r
	\end{equation}
	for each $n \in \N$.
	For $n=1$, \eqref{aux-eq-Trevisan-proof} holds by \eqref{aux-2}. For the induction step from $n$ to $n+1$, fix $r \leq t_1 < \dots < t_n< t_{n+1}$ and functions $h_i, i \in \{1,\dots,n+1\},$ as specified in the definition of $\Hscr^{(n+1)}_r$ above, and set
	\begin{equation*}
		\vrho := \frac{\Pi_{i=1}^nh_i(\pi^r_{t_i})}{\mathbb{E}_{P^1}[\Pi_{i=1}^nh_i(\pi^r_{t_i})] },
	\end{equation*}
	where the denominator is greater or equal to $\Pi_{i=1}^nc_i>0$. Note that 
	\begin{equation}\label{aux-c}
		\frac{1}{c} \leq \vrho \leq c\text{ pointwise for some }c>1,
	\end{equation} and $\mathbb{E}_{P^i}[\vrho] = 1$ for $i \in \{1,2\}$ by the induction hypothesis.
	Since for every $f \in \Bscr^+_b(\R^d)$ we have
	$$\int_{\Omega_r}f(\pi^r_{t_n}(w))\,(\vrho P^i)(dw) = \bigg[\int_{\Omega_r} \big(\Pi_{i=1}^nh_i(\pi^r_{t_i}(w))\big) f(\pi^r_{t_n}(w))\,P^i(dw)\bigg]\big(\mathbb{E}_{P^1}[\Pi_{i=1}^nh_i(\pi^r_{t_i})]\big)^{-1},$$
	and since the induction hypothesis implies that these terms are equal for $i\in \{1,2\}$, it follows that 
	\begin{equation}\label{proof-3.3-eq}
		(\vrho\,P^1) \circ (\pi^r_{t_n})^{-1} = (\vrho\,P^2) \circ (\pi^r_{t_n})^{-1}.
	\end{equation}
	By \cite[Lem.2.6]{Trevisan16}, the path measures $(\vrho P^i)\circ (\Pi^r_{t_n})^{-1}$, $i \in \{1,2\}$, on $\Bscr(\Omega_{t_n})$ solve $\eqref{linMGP}$, with $\mu^{t_n,\mu^{s,\zeta}_{t_n}}_t$ replacing $\mu_t$, from time $t_n$ and, by \eqref{proof-3.3-eq}, with identical initial condition. Consequently, their curves of one-dimensional time marginals $\eta^i = (\eta^i_t)_{t \geq t_n}:= ((\vrho\,P^i) \circ (\pi^r_{t})^{-1})_{t\geq t_n}$, $i \in \{1,2\}$, solve \eqref{LFPE}, with $\mu^{t_n,\mu^{s,\zeta}_{t_n}}_t= \mu^{s,\zeta}_t$ replacing $\mu_t$. For any $A \in \Bscr(\R^d)$ and $t \geq t_n$, we have by \eqref{aux-2}
	\begin{equation*}
		\eta^i_t(A) =  \int_{\Omega_r} \vrho(w)\mathds{1}_A(\pi^r_t(w))\,P^i(dw) \leq c \,P_t^i(A) = c\,\mu^{s,\zeta}_t, \quad i \in \{1,2\},
	\end{equation*}
	for $c$ as in \eqref{aux-c}, and consequently $\eta^i \in \Ascr_{t_n,\leq}(\mu^{t_n,\mu^{s,\zeta}_{t_n}})$. Similarly,
	$\eta^i_{t}(A) \geq \frac 1 c \mu^{s,\zeta}_t(A)$ for all $A\in \Bscr(\R^d)$ and $t \geq t_n$. In particular, for $t=t_n$, it follows that $\eta^i_{t_n} = g \, \mu^{s,\zeta}_{t_n}$ for some measurable $g: \R^d \to \R_+$ such that $\frac 1 {c_0} \leq g \leq c_0$ pointwise for some $c_0 >1$, and $\int_{\R^d} g(x)\,\mu^{s,\zeta}_{t_n} = 1$. By the assumption, Lemma \ref{lem:Michael-observation} and Lemma \ref{lem:uniqueness-extends} (ii), we obtain $(\eta^1_t)_{t \geq t_n} = (\eta^2_t)_{t \geq t_n}$, so in particular $\eta^1_{t_{n+1}} = \eta^2_{t_{n+1}}$. Now we conclude
	\begin{align*}
		\frac{\mathbb{E}_{P^i}\big[\Pi_{i=1}^{n+1}h_i(\pi^r_{t_i})\big]}{\mathbb{E}_{P^1}\big[\Pi_{i=1}^{n}h_i(\pi^r_{t_i})\big]}
		= \int_{\Omega_r} \vrho(w)&h_{n+1}(\pi^r_{t_{n+1}}(w)) \,P^i(dw)=\int_{\R^d} h_{n+1}(x)\,\eta^i_{t_{n+1}}(dx) 
	\end{align*}
	for $i \in \{1,2\}$, and conclude 
	$$
	\mathbb{E}_{P^1}\big[\Pi_{i=1}^{n+1}h_i(\pi^r_{t_i})\big] = \mathbb{E}_{P^2}\big[\Pi_{i=1}^{n+1}h_i(\pi^r_{t_i})\big],
	$$
	which gives \eqref{aux-eq-Trevisan-proof} for $n+1$, and hence completes the proof of the claim.
	
	Consequently, for any such initial datum $(s,\zeta)$, solutions to the corresponding nonlinear martingale problem \eqref{MGP} with one-dimensional time marginals $(\mu^{s,\zeta}_t)_{t \geq s}$ are unique. Now the uniqueness part of the assertion follows from the equivalence of \eqref{MGP} and \eqref{DDSDE} mentioned in Remark \ref{rem:known-relations-of-problems} (i). 
\end{proof}

Now we state and prove our main result of this paper on the existence of nonlinear Markov processes with prescribed one-dimensional time marginals, which are given as a solution flow to a nonlinear FPKE of type \eqref{NLFPE}.
\begin{theorem}\label{theorem1}
	Let $\Pscr_0 \subseteq\Pscr$ and assume that $\{\mu^{s,\zeta}\}_{(s,\zeta)\in \R_+\times \Pscr_0}$ is a solution flow to \eqref{NLFPE} such that $\mu^{s,\zeta} \in M^{s,\zeta}_{\mu^{s,\zeta},\text{ex}}$ for each $(s,\zeta) \in \R_+\times \Pscr_0$.
	Then the path laws $\mathbb{P}_{s,\zeta} := \mathcal{L}_{X^{s,\zeta}}$, $(s,\zeta) \in \R_+\times \Pscr_0$, of the unique weak solution to \eqref{DDSDE} with one-dimensional time marginals $(\mu^{s,\zeta}_t)_{t \geq s}$ from Theorem \ref{aux-prop-ex+uniqu} constitute a nonlinear Markov process.
\end{theorem}

Combining Theorem \ref{theorem1} and Lemma \ref{lem:Michael-observation}, one obtains the following corollary, which provides a convenient formulation to be checked in applications, compare with Section \ref{sect:Appl} below.

\begin{kor}\label{cor:alternative-form-main-thm}
	Let $\Pscr_0 \subseteq\Pscr$ and assume that $\{\mu^{s,\zeta}\}_{(s,\zeta)\in \R_+\times \Pscr_0}$ is a solution flow to \eqref{NLFPE} such that for every $(s,\zeta) \in \R_+\times \Pscr_0$,
	$\mu^{s,\zeta}$ is the unique solution to \eqref{LFPE}, with $\mu^{s,\zeta}_t$ replacing $\mu_t$, in $\Ascr_{s,\leq}(\mu^{s,\zeta})$ from $(s,\zeta)$.
	Then the path laws $\mathbb{P}_{s,\zeta} := \mathcal{L}_{X^{s,\zeta}}$, $(s,\zeta) \in \R_+\times \Pscr_0$, of the unique weak solution to \eqref{DDSDE} with one-dimensional time marginals $(\mu^{s,\zeta}_t)_{t \geq s}$ from Theorem \ref{aux-prop-ex+uniqu} constitute a nonlinear Markov process.
\end{kor}

\begin{proof}[Proof of Theorem \ref{theorem1}]
	By Theorem \ref{aux-prop-ex+uniqu}, the family $(\mathbb{P}_{s,\zeta})_{(s,\zeta)\in \R_+\times \Pscr_0}$ is such that for all $(s,\zeta) \in \R_+\times \Pscr_0$, the following holds:
	\begin{enumerate}
		\item [(i)] $\mathbb{P}_{s,\zeta} \in \Pscr(\Omega_s)$ and $\mathbb{P}_{s,\zeta}\circ (\pi^s_t)^{-1} = \mu^{s,\zeta}_t$ for all $t \geq s$,
		\item[(ii)] $\mathbb{P}_{s,\zeta}$ solves \eqref{MGP}, as well as \eqref{linMGP}, with $\mu^{s,\zeta}_t$ replacing $\mu_t$, from $(s,\zeta)$,
		\item[(iii)] $\mathbb{P}_{s,\zeta}$ is the path law of the unique weak solution to \eqref{DDSDE} with one-dimensional time marginals $(\mu^{s,\zeta}_t)_{t \geq s}$.
	\end{enumerate}
	To prove the nonlinear Markov property, let $0\leq s \leq r \leq t$ and $\zeta \in \Pscr_0$. Disintegrating $\mathbb{P}_{r,\mu^{s,\zeta}_r}$ with respect to $\mu^{s,\zeta}_r$ yields, as already mentioned in Remark \ref{rem:Markov-prop-extends-to-fct} (i),
	\begin{equation}\label{aux-main-proof-new}
		\mathbb{P}_{r,\mu^{s,\zeta}_r}(\cdot) = \int_{\mathbb{R}^d} p_{(s,\zeta),(r,y)}(\cdot)\,\mu^{s,\zeta}_r(dy)
	\end{equation}
	as  measures on $\Bscr(\Omega_r)$,
	where the $\mu^{s,\zeta}_r$-almost surely determined family $p_{(s,\zeta),(r,y)}$, $y\in \R^d$, of Borel probability measures on $\Omega_r$ is as in Definition \ref{def:NL-Markov-process}. By \cite[Prop.2.8]{Trevisan16}, for $\mu^{s,\zeta}_r$-a.e. $y \in \R^d$, $p_{(s,\zeta),(r,y)}$ solves \eqref{linMGP}, with $\mu^{r,\mu^{s,\zeta}_r}_t$ replacing $\mu_t$, from $(r,\delta_y)$. Hence, for any $\vrho \in \Bscr_b^+(\R^d)$ with $\int_{\R^d} \vrho(x)\, \mu^{s,\zeta}_r(dx) =1$, the measure $\mathbb{P}_\vrho \in \Pscr(\Omega_r)$, 
	\begin{equation}\label{def:P_rho}
		\mathbb{P}_\vrho := \int_{\mathbb{R}^d} p_{(s,\zeta),(r,y)}\,\vrho(y)\,\mu^{s,\zeta}_r(dy),
	\end{equation}
	solves the same linearized martingale problem with initial datum $(r,\vrho \,\mu^{s,\zeta}_r)$. Let $n \in \N$, $s \leq t_1 < \dots < t_n \leq r$, $h \in \mathcal{B}^+_b((\R^d)^n)$ such that $a^{-1} \leq h \leq a$ for some $a>1$, and let $\tilde{g}:\R^d \to \R_+$ be the bounded, $\mu^{s,\zeta}_r$-a.s. uniquely determined map such that
	$$\mathbb{E}_{s,\zeta}\big[h(\pi^s_{t_1},\dots,\pi^s_{t_n})|\sigma(\pi^s_r)\big]= \tilde{g}(\pi^s_r)\quad \mathbb{P}_{s,\zeta}-a.s.$$
	Let $g:=c_0 \tilde{g}$, where $c_0>0$ is such that $\int_{\R^d} g(x)\,\mu^{s,\zeta}_r(dx) = 1$, and let $\mathbb{P}_g$ be as in \eqref{def:P_rho}, with $g$ replacing $\vrho$, with initial condition $(r,g\, \mu^{s,\zeta}_r)$. 
	Also, consider $\theta: \Omega_s \to \R$, $\theta := c_0 h(\pi^s_{t_1},\dots,\pi^s_{t_n})$, i.e. $\mathbb{E}_{s,\zeta}[\theta] = 1$. Set $$\mathbb{P}^\theta := (\theta \,\mathbb{P}_{s,\zeta})\circ (\Pi^s_r)^{-1}.$$
	Note that $g(\R^d), \theta(\Omega_s) \subseteq [a^{-1}c_0,ac_0]$ $\mu^{s,\zeta}_r$-a.s., so in particular $g\,\mu^{s,\zeta}_r \sim \mu^{s,\zeta}_r$.
	By \cite[Prop.2.6]{Trevisan16}, also $\mathbb{P}^\theta$ solves \eqref{linMGP}, with $\mu^{s,\zeta}_t$ replacing $\mu_t$, and its initial datum is $(r,g \mu^{s,\zeta}_r)$, since for all $A\in \Bscr(\R^d)$
	\begin{align*}
		\mathbb{P}^\theta \circ (\pi^r_r)^{-1}(A) &= \int_{\Omega_s} \mathds{1}_A(\pi^s_r(w))\theta(w)\, \mathbb{P}_{s,\zeta}(dw) \\&= \int_{\Omega_s} \mathds{1}_A(\pi^s_r(w))g(\pi^s_r(w))\, \mathbb{P}_{s,\zeta}(dw) =\int_A g(y)\,\mu^{s,\zeta}_r(dy).
	\end{align*}
	In particular, both one-dimensional time marginal curves $(\mathbb{P}_g\circ (\pi^r_t)^{-1})_{t \geq r}$ and $(\mathbb{P}^\theta\circ (\pi^r_t)^{-1})_{t \geq r}$ solve \eqref{LFPE}, with $\mu^{s,\zeta}_t$ replacing $\mu_t$, from $(r,g\,\mu^{s,\zeta}_r)$. Moreover, 
	\begin{equation}\label{aux-eq-mainproof}
		\mathbb{P}^\theta \circ (\pi^r_t)^{-1},\,\mathbb{P}_g\circ (\pi^r_t)^{-1}\leq ac_0\,\mu^{r,\mu^{s,\zeta}_r}_t, \quad \forall t \geq r,
	\end{equation}
	i.e. both these one-dimensional time marginal curves belong to $\Ascr_{r,\leq}(\mu^{r,\mu^{s,\zeta}_r})$. 
	Indeed, \eqref{aux-eq-mainproof} can be seen as follows. For all $t \geq r$ and $A\in \Bscr(\R^d)$,
	\begin{align*}
		\mathbb{P}^\theta\circ(\pi^r_t)^{-1}(A) &= \int_{\Omega_s} \theta(w)\mathds{1}_{A}(\pi^s_t(w))\, \mathbb{P}_{s,\zeta}(dw)\\&\leq c_0a \int_{\Omega_s} \mathds{1}_A(\pi^s_t(w))\,\mathbb{P}_{s,\zeta}(dw)= c_0a\,\mu^{s,\zeta}_t(A) = c_0a \, \mu^{r,\mu^{s,\zeta}_r}_t(A).
	\end{align*}
	Similarly, by \eqref{aux-main-proof-new},
	\begin{align*}
		&\mathbb{P}_g\circ(\pi^r_t)^{-1}(A) = \int_{\R^d} p_{(s,\zeta),(r,y)}(\pi^r_t \in A)g(y)\,\mu^{s,\zeta}_r(dy) \leq ac_0 \,\mathbb{P}_{r,\mu^{s,\zeta}_r}(\pi^r_t \in A)= ac_0\, \mu^{r,\mu^{s,\zeta}_r}_t(A).
	\end{align*}
	Hence by the assumption, Lemma \ref{lem:Michael-observation} and Lemma \ref{lem:uniqueness-extends} (ii)
	\begin{equation*}
		\mathbb{P}_g\circ (\pi^r_t)^{-1} = \mathbb{P}^\theta \circ (\pi^r_t)^{-1},\quad \forall t \geq r,
	\end{equation*}
	and therefore for $t \geq r$ and $A \in \mathcal{B}(\R^d)$
	\begin{align*}
		\mathbb{E}_{s,\zeta}\big[h(\pi^s_{t_1},\dots,\pi^s_{t_n})\mathds{1}_{\pi^s_t \in A}\big] &= c_0^{-1}\mathbb{P}^\theta\circ (\pi^r_t)^{-1}(A) = c_0^{-1}\mathbb{P}_g\circ (\pi^r_t)^{-1}(A) \\&= c_0^{-1} \int _{\Omega_s}p_{(s,\zeta),(r,\pi^s_r(\omega))}(\pi^r_t \in A)g(\pi^s_r(\omega))\,\mathbb{P}_{s,\zeta}(d\omega) \\&= \int_{\Omega_s}p_{(s,\zeta),(r,\pi^s_r(\omega))}(\pi^r_t\in A)h(\pi^s_{t_1}(\omega),\dots,\pi^s_{t_n}(\omega)) \,\mathbb{P}_{s,\zeta}(d\omega).
	\end{align*}
	Here we used the $\sigma(\pi^s_r)$-measurability of $\Omega_s \ni \omega \mapsto p_{(s,\zeta),(r,\pi^s_r(\omega))}(\pi^r_t\in A)$ for the final equality. By a monotone class-argument, \eqref{Markov-prop} follows.
\end{proof}
Since the nonlinear Markov property is obviously always fulfilled for $s=r$, the following corollary follows from the proof of Theorem \ref{theorem1}. We will make use of it in the examples in Section \ref{sect:Appl} to allow singular initial data (e.g., Dirac measures) for nonlinear Markov processes. 

\begin{kor}\label{cor:after-main-thm}
	Let $\mathfrak{P}_0 \subseteq \Pscr_0 \subseteq \Pscr$ and let $\{\mu^{s,\zeta}\}_{(s,\zeta) \in \R_+\times \Pscr_0}$ be a solution flow to \eqref{NLFPE} such that $\mu^{s,\zeta}_t \in \mathfrak{P}_0$ for all $0\leq s \leq t, \zeta \in \mathfrak{P}_0$ (then $\{\mu^{s,\zeta}\}_{(s,\zeta)\in \R_+\times
		\mathfrak{P}_0}$ is a solution flow) and $\mu^{s,\zeta}_t \in \mathfrak{P}_0$ for all $0\leq s < t, \zeta \in \Pscr_0$. Also assume $\mu^{s,\zeta} \in M^{s,\zeta}_{\mu^{s,\zeta}, \text{ex}}$ for all $(s,\zeta) \in \R_+\times\mathfrak{P}_0$.
	
	Then there exists a nonlinear Markov process $(\mathbb{P}_{s,\zeta})_{(s,\zeta) \in \R_+\times \Pscr_0}$ such that $\mathbb{P}_{s,\zeta} \circ (\pi^s_t)^{-1} = \mu^{s,\zeta}_t$ for every $(s,\zeta) \in \R_+\times \Pscr_0$ and $t \geq s$, consisting of path laws of weak solutions to the corresponding \eqref{DDSDE}.
	Moreover, for $\zeta \in \mathfrak{P}_0$, $\mathbb{P}_{s,\zeta}$ is the path law of the unique weak solution to \eqref{DDSDE} with one-dimensional time marginals $(\mu^{s,\zeta}_t)_{t \geq s}$.
\end{kor}
A typical application of Corollary \ref{cor:after-main-thm} is as follows: $\mathfrak{P}_0 = \Pscr_a \cap L^\infty$, $\Pscr_0 = \Pscr$, one has a solution flow $\{\mu^{s,\zeta}\}_{(s,\zeta) \in \R_+\times \Pscr}$ to \eqref{NLFPE}, every solution $(\mu_t)_{t \geq s}$ to \eqref{NLFPE} is absolutely continuous with respect to Lebesgue measure with bounded density for every $t>s$, even if the initial datum $\mu_s = \zeta$ is a singular probability measure (e.g., a Dirac measure) (also called $L^1-L^\infty$-regularization), and for every initial measure $\zeta \in \Pscr_a\cap L^\infty$, solutions to \eqref{NLFPE} are unique in $\bigcap_{T>s}L^\infty((s,T)\times \R^d)$. See the second part of (i) in Section \ref{subsect:examples-nemytskii} below for a concrete example.

Similarly to Corollary \ref{cor:alternative-form-main-thm}, by Lemma \ref{lem:Michael-observation}, one can alternatively formulate Corollary \ref{cor:after-main-thm} in terms of the restricted uniqueness of the corresponding linearized FPKEs.

Now we give the postponed proof of Lemma \ref{lem:uniqueness-extends}.
\begin{proof}[Proof of Lemma \eqref{lem:uniqueness-extends}]
	\begin{enumerate}
		\item [(i)] Assume the claim is wrong, i.e. there is $\nu \sim \zeta_0$, $\nu \in \Pscr$, such that there are two solutions $(\eta_t^i)_{t \geq s_0}$ with $\eta^i_{s_0} =\nu$, $i \in \{1,2\}$, and $\eta^1_r \neq \eta^2_r$ for some $r >s_0$.  By the superposition principle, there are probability measures $P^i$ on $\mathcal{B}(\Omega_{s_0})$ with one-dimensional time marginal curves $P^i_t =\eta^i_t$, $t \geq s_0$, which solve the corresponding linear martingale problem corresponding to the linear FPKE from the assumption. Disintegrating each $P^i$ with respect to $\pi^{s_0}_{s_0}$ yields two probability kernels $\theta^i: \R^d \times \Bscr(\Omega_{s_0})\ni(y,D) \mapsto \theta^i_y(D)$ from $\R^d$ to $\Bscr(\Omega_{s_0})$ which are uniquely determined for $\nu$-a.e. (hence also $\zeta_0$-a.e.) $y \in R^d$, and for $\nu$-a.e. (hence $\zeta_0$-a.e.) $y \in \R^d$, $\theta^i_y$ solves the linear martingale problem with initial datum $(s_0,\delta_y)$, see \cite[Prop.2.8]{Trevisan16}. Since
		\begin{equation*}
			\int_{\mathbb{R}^d} \theta^1_y \circ (\pi^{s_0}_r) ^{-1} \,\nu(dy) =P^1_r = \eta^1_r \neq \eta^2_r = P^2_r = \int_{\mathbb{R}^d} \theta^2_y \circ (\pi^{s_0}_r) ^{-1} \,\nu(dy) ,
		\end{equation*}
		there is $A\in \Bscr(\R^d)$ such that (w.l.o.g.) the Borel set $E:= \{y \in \R^d: \theta^1_y(\pi^{s_0}_r \in A) > \theta^2_y(\pi^{s_0}_r \in A)\}$ has strictly positive $\nu$- (hence strictly positive $\zeta_0$-) measure. For $i \in \{1,2\}$, set
		\begin{equation*}
			\tilde{\theta}^i_y := \begin{cases}
				\theta^i_y, \quad y \in E,\\
				\theta^1_y, \quad y\in E^c.
			\end{cases}
		\end{equation*}
		Clearly, $\tilde{\theta}^i$ are again probability kernels from $\R^d$ to $\Bscr(\Omega_r)$, and the probability measures on $\Bscr(\Omega_{s_0})$ given by $\int \tilde{\theta}^i_y\, \zeta_0(dy)$ both solve the linear martingale problem with initial datum $(s_0,\zeta_0)$. Hence the curves $(\nu^i_t)_{t \geq {s_0}}$,
		\begin{equation*}
			\nu^i_t := \int_{\mathbb{R}^d}\tilde{\theta}^i_y \circ (\pi^{s_0}_t)^{-1} \, \zeta_0(dy),
		\end{equation*}
		solve the linear FPKE from the assertion with initial datum $(s_0,\zeta_0)$, and for $A$ and $r$ as above, we have
		\begin{align*}
			\nu^1_r(A) &= \int_E \theta^1_y \circ (\pi^{s_0}_r)^{-1}(A) \,\zeta_0(dy) + \int_{E^c} \theta^1_y \circ (\pi^{s_0}_r)^{-1}(A) \, \zeta_0(dy) \\&> \int_E \theta^2_y \circ (\pi^{s_0}_r)^{-1}(A) \, \zeta_0(dy) + \int_{E^c} \theta^1_y \circ (\pi^{s_0}_r)^{-1}(A) \, \zeta_0(dy) = \nu^2_r(A),
		\end{align*}
		i.e. $\nu^1\neq \nu^2$, which contradicts the assumption.
		\item[(ii)] We proceed as in (i), but assume in addition that $\nu = g \,\zeta_0$ with $g$ as in the assertion and $\eta^i \in \Ascr_{s_0,\leq}(\nu^{{s_0},\zeta_0})$ for $i \in \{1,2\}$. Then, denoting by $|\frac 1 g|_{\infty}$ the $L^\infty$-norm with respect to either of the equivalent measures $\zeta_0$ and $g\, \zeta_0$ (note that we even assume $\delta \leq g$ pointwise), we have
		\begin{align*}
			\nu^i_r &\leq \big|1/g\big|_{\infty} \bigg(\int_{\mathbb{R}^d} \theta^i_y \circ (\pi^{s_0}_r)^{-1}\,g(y)\,\zeta_0(dy)  +\int_{\mathbb{R}^d} \theta^1_y \circ (\pi^{s_0}_r)^{-1} \,g(y)\,\zeta_0(dy)\bigg) \\& = \big|1/g \big|_{\infty}(\eta^i_r+\eta^1_r),
		\end{align*}
		and since $\eta^i \in \Ascr_{s,\leq}(\nu^{{s_0},\zeta_0})$, also $\nu^i \in \Ascr_{{s_0},\leq}(\nu^{s,\zeta_0})$, which contradicts the assumption in (ii).
	\end{enumerate}
\end{proof}

\begin{rem}
	Due to Lemma \ref{lem:uniqueness-extends} (ii) and Lemma \ref{lem:Michael-observation}, the assumption of Corollary \ref{cor:after-main-thm} can slightly be further generalized as follows:
	
	Let $\mathfrak{P}_0 \subseteq \Pscr_0 \subseteq \Pscr$ and let $\{\mu^{s,\zeta}\}_{(s,\zeta) \in \R_+\times \Pscr_0}$ be a solution flow to \eqref{NLFPE} such that $\mu^{s,\zeta}_t \in \mathfrak{P}_0$ for all $0\leq s \leq t, \zeta \in \mathfrak{P}_0$ (then $\{\mu^{s,\zeta}\}_{(s,\zeta)\in \R_+\times
		\mathfrak{P}_0}$ is a solution flow) and $\mu^{s,\zeta}_t \in \mathfrak{P}_0$ for all $0\leq s < t, \zeta \in \Pscr_0$. Also assume 
	that for each $(s,\zeta) \in \R_+\times \mathfrak{P}_0$, there is $\mathfrak{P}_0 \ni \eta = g\, \zeta$ with $g$ bounded from above and from below by a strictly positive constant, such that solutions to \eqref{LFPE}, with $\mu^{s,\zeta}$ replacing $\mu_t$, are unique in $\Ascr_{s,\leq}(\mu^{s,\zeta})$ from $(s, g\, \zeta)$.
\end{rem}

\subsection{Implications for the special case of linear FPKEs and linear Markov processes}\label{subsect:lin-special case}
Since linear FPKEs are special cases of nonlinear FPKEs, our results from the previous subsection particularly apply to the theory of linear FPKEs and linear Markov processes. It turns out that even in this classical situation, our general nonlinear theorems contain new linear results, which we explicitly state and discuss in this subsection. More precisely, we consider Theorem \ref{aux-prop-ex+uniqu} and Theorem \ref{theorem1} for the case that \eqref{NLFPE} is linear, i.e. its coefficients are independent of their measure variable.

As before, we use the notation $M^{s,\zeta}$ and $M^{s,\zeta}_{\text{ex}}$ for the set of solutions to this equation from $(s,\zeta)$ and the set of extremal points in $M^{s,\zeta}$, respectively. In this case, the associated stochastic equation is an SDE whose coefficients are not distribution-dependent.

\begin{prop}[Linear version of Thm.\ref{aux-prop-ex+uniqu}]\label{prop:linear-prop-from-Prop3.4.}
	Let $\Pscr_0 \subseteq\Pscr$ and assume that $\{\mu^{s,\zeta}\}_{(s,\zeta)\in \R_+\times \Pscr_0}$ is a solution flow to a linear FPKE such that $\mu^{s,\zeta} \in M^{s,\zeta}_{ex}$ for each $(s,\zeta) \in \R_+\times \Pscr_0$. Then for every $(s,\zeta) \in \R_+\times \Pscr_0$, there is a unique weak solution $X^{s,\zeta}$
	to the corresponding SDE with initial condition $(s,\zeta)$ and one-dimensional time marginals equal to $(\mu^{s,\zeta}_t)_{t \geq s}$.
\end{prop}
Similarly, as in the nonlinear case, replacing the assumption $\mu^{s,\zeta} \in M^{s,\zeta}_{\text{ex}}$ for all $(s,\zeta) \in \R_+\times \Pscr_0$ by the assumption that each $\mu^{s,\zeta}$ is the unique element in $M^{s,\zeta}\cap \Ascr_{s,\leq}(\mu^{s,\zeta})$, gives an equivalent formulation of Proposition \ref{prop:linear-prop-from-Prop3.4.}.
\begin{rem}\label{rem:flow-select-linear-case-Rehmeier}
	
	We would like to mention that a general method to construct solution flows to linear FPKEs in non-unique situations consisting of extremal points in $M^{s,\zeta}$, $(s,\zeta) \in \R_+\times \Pscr_0$, has been presented in \cite{Rehmeier_Flow-JEE,Rehmeier-nonlinear-flow-JDE}. Indeed, although it was not pointed out in these papers, the extremality of the members of the constructed flows can easily be read off from the iterative selection presented in the proof of Theorem 3.1. in \cite{Rehmeier-nonlinear-flow-JDE} (which contains the results of \cite{Rehmeier_Flow-JEE}).
	Hence, for any solution flow $\{\mu^{s,\zeta}\}_{(s,\zeta) \in \R_+\times \Pscr_0}$ selected from $(M^{s,\zeta})_{(s,\zeta) \in \R_+\times \Pscr_0}$  as in \cite{Rehmeier-nonlinear-flow-JDE}, there exists a unique family of weak solutions $(X^{s,\zeta})_{(s,\zeta) \in \R_+\times \Pscr_0}$ with $\mathcal{L}_{X^{s,\zeta}_t} = \mu^{s,\zeta}_t$ for all $0\leq s \leq t$ and $\zeta \in \Pscr_0$ to the corresponding SDE.
\end{rem}

For the special case of a linear FPKE, the main result of our paper, Theorem \ref{theorem1}, reads as follows.
\begin{prop}[Linear version of Theorem \ref{theorem1}]\label{prop:linear-version-mainthm}
	Let $\Pscr_0 =\Pscr$, and assume that $\{\mu^{s,\zeta}\}_{(s,\zeta) \in \R_+\times \Pscr}$ is a solution flow to a linear FPKE such that $\mu^{s,\zeta} \in M^{s,\zeta}_{\text{ex}}$ for each $(s,\zeta) \in \R_+\times \Pscr$. Then the family of the path laws $\mathbb{P}_{s,\zeta} = \Lscr_{X^{s,\zeta}}$, $(s,\zeta) \in \R_+\times \Pscr$, of the unique weak solutions to the corresponding SDE with $\Lscr_{X^{s,\zeta}_t}= \mu^{s,\zeta}_t$, given by Proposition \ref{prop:linear-prop-from-Prop3.4.}, is a nonlinear Markov process. 
	Moreover, if $x \mapsto \mu^{s,\delta_x}_t$ is measurable for all $0\leq s \leq t$ and the flow also fulfills the Chapman-Kolmogorov equations \eqref{CK}, then $(\mathbb{P}_{s,\zeta})_{(s,\zeta) \in \R_+\times \Pscr}$ is a linear Markov process, i.e. one has \eqref{intro:Markov-prop} and $\mathbb{P}_{s,\zeta} = \int_{\R^d} \mathbb{P}_{s,\delta_y}\,\zeta(dy)$ for all $(s,\zeta)\in \R_+\times \Pscr$.
\end{prop}

An equivalent formulation, i.e. the linear special case of Corollary \ref{cor:alternative-form-main-thm}, is obtained by replacing the assumption $\mu^{s,\zeta} \in M^{s,\zeta}_{\text{ex}}$ for all $(s,\zeta) \in \R_+\times \Pscr$ by the assumption that each $\mu^{s,\zeta}$ is the unique element in $M^{s,\zeta}\cap \Ascr_{s,\leq}(\mu^{s,\zeta})$.
As a consequence of Proposition \ref{prop:linear-version-mainthm} and Remark \ref{rem:flow-select-linear-case-Rehmeier}, the solution flows to linear FPKEs selected in \cite{Rehmeier-nonlinear-flow-JDE} with $\Pscr_0 = \Pscr$ give rise to nonlinear Markov processes, consisting of the path laws $\mathbb{P}_{s,\zeta}$ of the unique weak solutions $X^{s,\zeta}$ to the corresponding SDE, given by Proposition \ref{prop:linear-prop-from-Prop3.4.}. In \cite[Sect.5]{Rehmeier-nonlinear-flow-JDE}, it was proven that at least for bounded and continuous coefficients, the solution flows constructed in \cite{Rehmeier-nonlinear-flow-JDE} with $\Pscr_0 = \Pscr$ are measurable in the sense of Proposition \ref{prop:linear-version-mainthm} and satisfy \eqref{CK}. Consequently, we have the following corollary. We stress that in general it is still an open question whether every solution flow with $\Pscr_0 = \Pscr$ to a linear FPKE satisfies the Chapman-Kolmogorov equations.

\begin{kor}\label{cor:final-cor-lin-sect}
	Let the coefficients of a linear FPKE be continuous in $x \in \R^d$ and bounded. Then, there exists a solution flow $\{\mu^{s,\zeta}\}_{(s,\zeta) \in \R_+\times \Pscr}$ to this equation such that $\mu^{s,\zeta} \in M^{s,\zeta}_{\text{ex}}$ for all $(s,\zeta)\in \R_+\times \Pscr$, and the path laws $\mathbb{P}_{s,\zeta} = \Lscr_{X^{s,\zeta}}$, $(s,\zeta) \in \R_+\times \Pscr$, of the unique weak solutions to the corresponding SDE with $\Lscr_{X^{s,\zeta}_t}= \mu^{s,\zeta}_t$, given by Proposition \ref{prop:linear-prop-from-Prop3.4.}, form a linear Markov process. 
\end{kor}

\begin{rem}
	\begin{enumerate}
		\item[(i)] Following the pioneering work \cite{Krylov_1973}, in \cite[Ch.12]{StroockVaradh2007}, linear Markovian selections are constructed directly for the linear martingale problem associated with the linear FPKE under the same assumptions on the coefficients as in Corollary \ref{cor:final-cor-lin-sect}. However, to us it seems that a strength of selecting on the level of the one-dimensional time marginals instead of the path laws is that, as stated in Propositions \ref{prop:linear-prop-from-Prop3.4.} and \ref{prop:linear-version-mainthm} and Corollary \ref{cor:final-cor-lin-sect}, we obtain in addition that the Markovian path laws are the unique SDE solutions with the prescribed one-dimensional time marginals given by the solution flow for the linear FPKE. To the best of our knowledge, such a uniqueness result for the members of Markovian selections has not been obtained by selecting directly on the level of the path laws.
		\item[(ii)] The connection between Markovian selections of path laws and their extremality in the convex space of all solutions to the corresponding martingale problem was studied in depth in \cite{SY80}. In particular, in Theorem 2.10. the authors prove that for time-homogeneous, continuous and bounded coefficients, the Markovian selections from \cite[Ch.12]{StroockVaradh2007} consist of extremal points of the respective sets of solutions to the linear martingale problem.
	\end{enumerate}
\end{rem}
We conclude this section by pointing out that Corollary \ref{cor:final-cor-lin-sect} applies for instance to continuity equations (i.e. linear FPKES with $a =0$) with bounded and merely continuous (in $x\in \R^d)$ vector fields $b$. It is well-known that this includes plenty of interesting ill-posed examples.

\section{Applications and examples}\label{sect:Appl}

\subsection{Well-posed nonlinear FPKEs}\label{sect:well-posed-eq}
Theorem \ref{theorem1} (more precisely, Corollary \ref{cor:alternative-form-main-thm}) particularly applies, if for any $(s,\zeta) \in \R_+\times\Pscr_0$ the nonlinear FPKE \eqref{NLFPE} has a unique solution $(\mu_t^{s,\zeta})_{t \geq s}$ from $(s,\zeta)$ such that $\mu^{s,\zeta}_t \in \Pscr_0$ for all $t \geq s$, which is also the unique solution to \eqref{LFPE}, with $\mu^{s,\zeta}$ replacing $\mu_t$, from $(s,\zeta)$. The first part is particularly true when the associated McKean-Vlasov equation \eqref{DDSDE} has a unique weak solution from $(s,\zeta)$ with one-dimensional time marginals in $\Pscr_0$. Such a uniqueness result was obtained in \cite{Funaki} for the case of Lipschitz-continuous coefficients on $\R^d\times \Pscr_p$, where for $p\geq 1$ 
$$\mathcal{P}_p := \bigg\{\zeta \in \mathcal{P}: \int_{\mathbb{R}^d}|x|^p\zeta(dx) <\infty\bigg\}$$
is equipped with the $p$-Wasserstein distance $d_{\mathbb{W}_p}$.
The following generalization to one-sided Lipschitz drift $b$ was proven in \cite{McKean-recent4}:
for $p \geq 2$, \eqref{DDSDE} has a unique weak solution $X^{s,\zeta}$ from $(s,\zeta) \in \R_+\times \Pscr_p$ with $\mu^{s,\zeta}_t:=\mathcal{L}_{X^{s,\zeta}_t} \in \Pscr_p$,  if there is a strictly positive, increasing continuous function $K$ on $\R_+$ such that for all $t \geq 0, x,y \in \R^d$ and $ \mu,\nu \in \Pscr_p$:
\begin{enumerate}
	\item [(A1)]$|\sigma(t,\mu,x)-\sigma(t,\nu,y)|^2 \leq K(t)\big(|x-y|^2+d_{\mathbb{W}_p}(\mu,\nu)^2\big)$,
	\item[(A2)] $2\langle b(t,\mu,x)-b(t,\nu,y),x-y\rangle \leq K(t)\big(|x-y|^2+d_{\mathbb{W}_p}(\mu,\nu)\big)$,
	\item[(A3)] $b$ is bounded on bounded sets in $\R_+\times \Pscr_p\times \R^d$, $b(t,\cdot, \cdot)$ is continuous on $ \Pscr_p\times \R^d$, and $$|b(t,0,\mu)|^p \leq K(t)\bigg(1+\int_{\R^d} |z|^p \,\mu(dz)\bigg).$$
\end{enumerate}
Under these assumptions $\mu^{s,\zeta}$ is also the unique solution to \eqref{LFPE} from $(s,\zeta)$, with $\mu^{s,\zeta}_t$ replacing $\mu_t$. Hence by Corollary \ref{cor:alternative-form-main-thm}, the path laws $\mathbb{P}_{s,\zeta} := \mathcal{L}_{X^{s,\zeta}}$, $(s,\zeta) \in \R_+\times \Pscr_p$, form a nonlinear Markov process with one-dimensional time marginals $\mu^{s,\zeta}_t$, and the latter uniquely determine $\mathbb{P}_{s,\zeta}$, $(s,\zeta)\in \R_+\times \mathcal{P}_p$.

Further well-posedness results for equations of type \eqref{DDSDE} were obtained in  \cite{McKean-recent5,HW20,McKean-recent1,RZ21,HW22}.
While in these situations it has to be checked separately whether also the linear equations \eqref{LFPE} are well-posed in the sense of Corollary \ref{cor:alternative-form-main-thm}, this is in general much easier than in the nonlinear case and thus often true in the above works. Then the path  laws of these unique weak solutions to \eqref{DDSDE} form a nonlinear Markov process.

\subsection{Restricted well-posed nonlinear FPKEs of Nemytskii-type}\label{subsect:examples-nemytskii}
We recall that we set $\Pscr_a := \{\mu \in \Pscr | \mu \ll dx\}$, and here we also use the notation $\Pscr_a^\infty := \{\mu \in \Pscr_a| \frac{d\mu}{dx} \in L^\infty(\R^d)\}$.

Without sufficient regularity in $(x,\mu) \in \R^d\times \Pscr$, e.g. Lipschitz-continuity with respect to a Wasserstein distance, or non-degeneracy for the diffusion coefficient, well-posedness for nonlinear (in fact, even for linear) FPKEs usually fails, but one might still be able to prove uniqueness in a subclass of solutions. Here we present such examples, giving rise to nonlinear Markov processes in the case of singular \textit{Nemytskii-type} coefficients, i.e. for $\mu \in \Pscr_a$ with $\mu = \frac{d\mu}{dx} dx$
\begin{equation}\label{Nemytskii-type coeff.}
	b(t,\mu,x) := \tilde{b}\bigg(t,\frac{d\mu}{dx}(x),x\bigg),\quad a(t,\mu,x) := \tilde{a}\bigg(t,\frac{d\mu}{dx}(x),x\bigg)
\end{equation}
for Borel coefficients $\tilde{b},\tilde{a}$ on $\R_+\times \R \times \R^d$ with values in $\R^d$ and $\R^{d\times d}$, respectively. Without further mentioning, we always consider the version of $\frac{d\mu}{dx}$ which is $0$ on those $x \in \R^d$ for which $\lim_{r \to 0}dx(B_r(0))^{-1}\mu(B_r(x))$ does not exist in $\R$ (here $B_r(0)$ denotes the Euclidean ball of radius $r >0$ centered at $x$). By Lebesgue's differentiation theorem, the set of such $x$ is a $dx$-zero set. Since $(\mu,y)\mapsto \frac{d\mu}{dx}(y)$ is $\mathcal{B}(\Pscr_a)\otimes \Bscr(\R^d)$-measurable by \cite[Sect.4.2.]{Grube-thesis}, it follows that $b$ and $a$ are product Borel measurable on $\R_+\times \Pscr_a\times \R^d$ for the weak topology on $\Pscr_a$.

\begin{enumerate}
	\item [(i)] \textbf{Generalized porous media equation.}
	Let 
	\begin{equation}\label{GPME-coeff}
		\tilde{b}(t,z,x) := b_0(z)D(x), \quad\tilde{a}(t,z,x):= \frac{\beta(z)}{z}\Id
	\end{equation}
	for Borel maps $b_0,\beta: \R \to \R$ and $D: \R^d \to \R$,
	where $\Id$ denotes the $d\times d$-identity matrix. In this case, \eqref{NLFPE} is a generalized porous media equation for the densities $u_t(\cdot) = \frac{d\mu_t}{dx}(\cdot)$, namely
	\begin{equation}\label{GPME}\tag{GPME}
		\partial_t u = \Delta \beta(u)-\divv\big(Db_0(u)u\big).
	\end{equation}
	A (distributional probability) solution to \eqref{GPME} from $(s,\zeta)\in \R_+\times \Pscr$ is a map $u: (s,\infty) \to \Pscr_a$ such that $t\mapsto u_t\,dx =:\mu_t$ is a weakly continuous solution to \eqref{NLFPE} in the sense of Definition \ref{def:all-single-eq.} (v) with $\lim_{t \to s}u_t\,dx = \zeta$ in the weak topology on $\Pscr$.
	Suppose that the following assumptions are satisfied.
	\begin{enumerate}
		\item [(B1)] $\beta(0)=0, \beta \in C^2(\R), \beta'\geq0$.
		\item[(B2)]  $b_0 \in C^1(\R)\cap C_b(\R), b\geq 0$.
		\item[(B3)] $D \in L^\infty(\R^d;\R^d), \divv D \in L^2_{\loc}(\R^d), (\divv D)^- \in L^\infty(\R^d)$.
		\item[(B4)] $\forall K \subset \R$ compact: $\exists\, \alpha_K >0$ with $|b_0(r)r-b_0(s)s| \leq \alpha_K |\beta(r)-\beta(s)|$ $\forall r,s \in K$.
	\end{enumerate}
	Then by \cite[Thm.2.2]{NLFPK-DDSDE5} for $\zeta \in \Pscr_0 := \Pscr_a^\infty$, there is a (distributional weakly continuous probability) solution $u^{s,\zeta}$ to \eqref{GPME} from $(s,\zeta)$ such that $u^{s,\zeta} \in \bigcap_{T>s}L^\infty((s,T)\times \R^d)$, $u^{s,\zeta}_t \in \Pscr_0$ for all $t \geq s$, and $\{u^{s,\zeta}\}_{(s,\zeta) \in \R_+\times \Pscr_0}$ is a solution flow (we identify $u^{s,\zeta}$ with $(\mu^{s,\zeta}_t)_{t \geq s}$, $\mu^{s,\zeta}_t := u^{s,\zeta}(t,x)dx$). Moreover, by \cite[Cor.4.2]{BR22}, the linearized equation \eqref{LFPE} associated with \eqref{GPME} with initial datum $(s,\zeta)$, with $\mu^{s,\zeta}_t$ replacing $\mu_t$, has $\mu^{s,\zeta}$ as its unique (weakly continuous probability) solution in $\bigcap_{T>s}L^\infty((s,T)\times \R^d)$.
	
	Therefore, Corollary \ref{cor:alternative-form-main-thm} applies and gives the existence of a nonlinear Markov process $(\mathbb{P}_{s,\zeta})_{(s,\zeta)\in \R_+\times \Pscr_a^\infty}$ with one-dimensional time marginals $\mathbb{P}_{s,\zeta}\circ (\pi^s_t)^{-1} = u^{s,\zeta}_t\,dx$, which is uniquely determined by the solution flow. More precisely, since by \cite[Cor.3.4]{BR22}, $u^{s,\zeta}$ is the unique (distributional weakly continuous probability) solution to \eqref{GPME} from $(s,\zeta)$ in $\bigcap_{T>s}L^\infty((s,T)\times \R^d)$, it follows from Theorem \ref{aux-prop-ex+uniqu} and Lemma \ref{lem:Michael-observation} that $\mathbb{P}_{s,\zeta}$ is the path law of the unique weak solution $X^{s,\zeta}$ to the corresponding McKean-Vlasov equation
	\begin{equation}\label{DDSDE-Nemytskii-special-case}
		dX_t = b_0\big(u_t(X_t)\big)D(X_t))dt + \sqrt{\frac{2\beta(u_t(X_t))}{u_t(X_t)}}dB_t,\,\,\, \mathcal{L}_{X_t}(dx) = u_t\,dx,\,\,\, t\geq s,\,\,\, \mathcal{L}_{X_s} = \zeta,
	\end{equation}
	under the constraint $u \in \bigcap_{T>s}L^\infty((s,T)\times \R^d)$. We stress that well-posedness for \eqref{GPME} or its linearized equations in a larger class than $\bigcap_{T>s}L^\infty((s,T)\times \R^d)$ is not known.
	
	The special case $\gamma_1 \leq \beta' \leq \gamma_2$ for $\gamma_i >0$ (which already implies (B4)) and $D = -\nabla \Phi$ for
	$$\Phi \in C^1(\R^d),\, \Phi \geq 1, \,\lim_{|x|\to \infty}\Phi(x) = \infty,\, \Phi^{-m} \in L^1(\R^d)\text{ for some }m \geq 2,$$
	called \textit{nonlinear distorted Brownian motion (NLDBM)}, was already studied in \cite{RRW20}.
	\\
	\\
	\textbf{Degenerate initial data.} 
	We continue this example with cases where the set of admissible initial data $\mathcal{P}_0$ includes singular measures, e.g. the Dirac measures. For such initial data, (restricted) uniqueness of solutions to \eqref{NLFPE} and \eqref{LFPE} is difficult to obtain, in particular in the singular Nemytskii-case (although there are positive results for pure-diffusion equations, cf. \cite{Pierre82}). However, if solutions started from singular measures are more regular after any strictly positive time, one can still obtain the nonlinear Markov property without any uniqueness for singular initial data.
	Let $b,a$ be as in \eqref{Nemytskii-type coeff.},\eqref{GPME-coeff}, and assume the following:
	\begin{enumerate}
		\item [(C1)] Assumptions (B1)-(B4) hold. In addition assume:
		\item[(C2)] $\beta'(r) \geq a|r|^{\alpha-1}$  and $|\beta(r)|\leq C|r|^\alpha$ for some $a,C >0$ and $\alpha \geq 1$.
		\item[(C3)] $D \in L^\infty\cap L^2(\R^d;\R^d)$, $\divv D \in L^2(\R^d)$, $\divv D \geq 0.$
	\end{enumerate}
	Then by \cite[Thm.5.2]{NLFPK-DDSDE5}, for each $(s,\zeta) \in \R_+ \times \Pscr$, there is a (weakly continuous probability) solution $u^{s,\zeta}$ to \eqref{GPME} from $(s,\zeta)$ such that $u^{s,\zeta} \in \bigcap_{r>s} L^\infty((r,\infty)\times \R^d)$. If $\zeta \in \Pscr_a^\infty$, even $u^{s,\zeta} \in L^\infty((s,\infty)\times \R^d)$ and in this case, since (B1)-(B4) hold, it follows as in the first part of this example that $u^{s,\zeta}$ is the unique weakly continuous distributional probability solution to \eqref{GPME} from $(s,\zeta)$ in $\bigcap_{T>s}L^\infty((s,T)\times \R^d)$. Consequently, since for any $s <r$ and $\zeta \in \Pscr$, both $(u^{s,\zeta}_t)_{t \geq r}$ and $(u^{r,u^{s,\zeta}_rdx}_t)_{t \geq r}$ solve \eqref{GPME} from $(r,u^{s,\zeta}_rdx)$ and belong to $\bigcap_{T>r}L^\infty((r,T)\times \R^d)$, it follows that $\{\mu^{s,\zeta}\}_{(s,\zeta) \in \R_+\times \Pscr}$, $\mu^{s,\zeta}_t := u^{s,\zeta}_tdx$, is a solution flow (in contrast to the first part of this example, this does not follow directly from the construction of $\mu^{s,\zeta}$ in \cite{NLFPK-DDSDE5}). Since for $s<r$ and $\zeta \in \Pscr$ we have $(u^{s,\zeta}_t)_{t \geq r} \in L^\infty((r,\infty)\times \R^d)$, it follows as in the first part of this example that the corresponding linearized equation with initial datum $(r,\mu^{s,\zeta}_r)$ and with $\mu^{s,\zeta}_t$ replacing $\mu_t$ has $\mu^{s,\zeta} = \mu^{r,\mu^{s,\zeta}_r}$ as its unique solution in $L^\infty((r,\infty)\times \R^d)$.
	
	Hence by Lemma \ref{lem:Michael-observation} and Corollary \ref{cor:after-main-thm}, with $\mathfrak{P}_0 = \Pscr_a^\infty$ and $\Pscr_0 = \Pscr$, there is a nonlinear Markov process $(\mathbb{P}_{s,\zeta})_{(s,\zeta) \in \R_+\times \Pscr}$ with one-dimensional time marginals $\mathbb{P}_{s,\zeta}\circ( \pi^s_t)^{-1} = \mu^{s,\zeta}_t$, consisting of path laws of weak solutions to \eqref{DDSDE-Nemytskii-special-case}. For  $\zeta \in \Pscr_a^\infty$, $\mathbb{P}_{s,\zeta}$ is the path law of the unique weak solution to \eqref{DDSDE-Nemytskii-special-case} with one-dimensional time marginals in $\bigcap_{T>s}L^\infty((s,T)\times \R^d)$ (as in the first part of this example).
	
	\item[(ii)] \textbf{Burgers' equation.}
	Consider Burgers' equation in $\R^1$, i.e. 
	\begin{equation*}
		\partial_t u = \frac{\partial^2u}{\partial^2x} -u \frac{\partial u}{\partial x},
	\end{equation*}
	which is one of the key examples McKean hinted at in his seminal work \cite{McKean1-classical}. Here we put forward his program by realizing classical solutions to Burgers' equation as one-dimensional time marginals of a nonlinear Markov process.
	Since we restrict attention to smooth pointwise solutions, we consider the equivalent equation
	\begin{equation}\label{eq:Burgers}\tag{BE}
		\partial_t u = \frac{\partial^2 u}{\partial^2 x} -\frac{1}{2}\frac{\partial u^2}{\partial x}
	\end{equation}
	instead, which is a nonlinear FPKE of Nemytskii-type with coefficients 
	\begin{equation}\label{coefficients-Burgers}
		b(t,\mu,x) :=\tilde{b}\bigg( \frac{d\mu}{dx}(x)\bigg) :=  \frac 1 2 \frac{d\mu}{dx}(x)\text{ and }a(t,\mu,x) := 1.
	\end{equation}
	For $(s,\zeta) \in \R_+\times \Pscr_a^\infty$, \eqref{eq:Burgers} has a unique smooth pointwise (i.e. strong) solution $u^{s,\zeta}$ on $(s,\infty)$ such that $u^{s,\zeta}_t  \to \zeta$ in $L^1(dx$) as $t \to s$, see for instance \cite{GR91-for-Burgers}. Furthermore, $0\leq u^{s,\zeta}_t \leq |\zeta|_{\infty}$ and 
	$$\int_{\R} u^{s,\zeta}_t(x)\,dx = \int_{\R} \zeta(x) \, dx$$
	for all $t \geq s$. Hence, $u^{s,\zeta}$ is in particular a distributional probability solution with initial datum $(s,\zeta)$ (as introduced in the first part of (i) above), and its uniqueness in the class of pointwise solutions implies that $\{u^{s,\zeta}\}_{(s,\zeta)\in \R_+\times \Pscr_a^\infty}$ is a solution flow for \eqref{eq:Burgers}. Since $u^{s,\zeta} \in L^\infty((s,\infty)\times \R)\cap L^1((s,T)\times \R)$ for all $T>s$, the superposition principle \cite[Thm.2.5]{Trevisan16} applies, i.e. there is a probabilistic weak solution $X^{s,\zeta}$ on $(s,\infty)$ to
	\begin{equation}\label{eq:stoch-Burgers}
		dX_t = \frac 1 2 u_t(X_t)dt+dB_t, \quad \mathcal{L}_{X_t}(dx) = u_t(x)\,dx,\quad  \mathcal{L}_{X_s} = \zeta,
	\end{equation}
	with $u = u^{s,\zeta}$.
	Since $u^{s,\zeta} \in L^\infty((s,\infty)\times \R)$, by \cite[Cor.4.2]{BR22} the linearized equation
	\begin{equation*}
		\partial_t v = \frac{\partial^2 v}{\partial^2 x}-\frac 1 2 \frac{\partial(u^{s,\zeta}v)}{\partial x},\quad t \geq s,\quad v_s = \zeta
	\end{equation*}
	has $v = u^{s,\zeta}$ as its unique weakly continuous probability solution in $\bigcap_{T >s}L^\infty((s,T)\times \R)$ for each $(s,\zeta) \in \R_+\times \Pscr_a^\infty$. More precisely, to see this one replaces $\tilde{b}$ from \eqref{coefficients-Burgers} by a bounded $C^1$ map which coincides with $\tilde{b}$ on $[0,\sup_{t\geq s,x\in \R^d}u^{s,\zeta}_t(x)]$. 
	
	Consequently, by Corollary \ref{cor:alternative-form-main-thm}, the path laws $\mathbb{P}_{s,\zeta}$ of the unique solutions $X^{s,\zeta}$, $(s,\zeta) \in \R_+\times \Pscr_a^\infty$, to \eqref{eq:stoch-Burgers} from $(s,\zeta)$ with time marginal densities in $\bigcap_{T>s}L^\infty((s,T)\times \R)$ form a uniquely determined nonlinear Markov process.
	\item[(iii)] \textbf{Distribution-dependent stochastic equations with Lévy noise.} For $\alpha\in (\frac 1 2,1)$ and $\zeta \in \Pscr_a^\infty$, consider the Cauchy problem for the fractional generalized porous media equation
	\begin{equation}\label{frac-GPME}
		\partial_t u = (-\Delta)^\alpha\beta(u)-\divv\big(Db(u)u\big),\quad t \geq s, \quad u_s(x)dx = \zeta(dx).
	\end{equation}
	Since the operator $(-\Delta)^\alpha$ is nonlocal, such equations are related to nonlocal Kolmogorov operators, see \cite{BR22-frac}. Associated with this non-local equation is the distribution-dependent SDE 
	\begin{equation}\label{DDSDE-frac}
		dX_t = b_0\big(u_t(X_t)\big)D(X_t)dt + \bigg(\frac{\beta(u_t(X_{t-}))}{u_t(X_{t-})}\bigg)^{\frac{1}{2\alpha}}dL_t,\quad \mathcal{L}_{X_t}(dx) = u_t(x)dx,\quad t\geq s,
	\end{equation}
	where $L$ is a $d$-dimensional isotropic $2s$-stable process with Lévy measure $dz/|z|^{d+2\alpha}$. A probabilistic weak solution $X$ to \eqref{DDSDE-frac} is defined as in Def.\eqref{def:all-single-eq.} (i) with $L$ instead of $B$, see \cite{BR22-frac} for details.
	Suppose the following assumptions hold.
	\begin{enumerate}
		\item[(D1)] $\beta \in C^\infty(\R)$, $\beta' >0$, $\beta(0) = 0$.
		\item[(D2)] $D \in L^\infty(\R^d;\R^d)$, $\divv D \in L^2_{\loc}(\R^d)$, $(\divv D)^- \in L^\infty(\R^d)$.
		\item[(D3)] $b_0 \in C_b(\R)\cap C^1(\R)$, $b \geq 0$.
	\end{enumerate}
	Then, by \cite[Thm.2.4,Thm.3.1]{BR22-frac}, for $(s,\zeta) \in \R_+ \times \Pscr_a^\infty$ there is a weakly continuous distributional solution $t \mapsto u^{s,\zeta}_t \in \Pscr_a^\infty$ to \eqref{frac-GPME} from $(s,\zeta)$ with the following properties. $\{u^{s,\zeta}\}_{(s,\zeta) \in \R_+\times \Pscr_a^\infty}$ is a solution flow in $\Pscr_0 =\Pscr_a^\infty$, and $u^{s,\zeta}$ is the unique weakly continuous probability solution to \eqref{frac-GPME} in $\bigcap_{T>s}L^\infty((s,T)\times \R^d)$. Moreover, by \cite[Thm.3.2]{BR22-frac}, it is also the unique solution to the linearized equation
	\begin{equation*}
		\partial_t v = (-\Delta)^\alpha\bigg(\frac{\beta(u^{s,\zeta})}{u^{s,\zeta}}v\bigg)-\divv\big(Db(u^{s,\zeta})v\big),\quad t\geq s, \quad v(s,\cdot) = \zeta,
	\end{equation*}
	in $\bigcap_{T>s}L^\infty((s,T)\times \R^d)$. Consequently, by a nonlocal version of the superposition principle from  \cite{SPpr-nonloc}, Corollary \ref{cor:alternative-form-main-thm} also applies in this nonlocal case and implies the existence of a nonlinear Markov process $(\mathbb{P}_{s,\zeta})_{(s,\zeta) \in \R_+ \times\Pscr_a^\infty}$, where $\mathbb{P}_{s,\zeta}$ is the path law of a càdlàg weak solution $X^{s,\zeta}$ to \eqref{DDSDE-frac} with one-dimensional time marginals $u_t^{s,\zeta}(x)dx$. By \cite[Thm.4.2]{BR22-frac}, $X^{s,\zeta}$ is unique under the constraint $u \in \bigcap_{T>s}L^\infty((s,T)\times \R^d)$. The nonlinear Markov property of the solutions to \eqref{DDSDE-frac} was already hinted at in \cite[Remarks 4.3,4.4]{BR22-frac} and is now finally indeed proved by the above.
	\item[(iv)] \textbf{Barenblatt solutions to the classical PME.} For the case of the classical porous media equation
	\begin{equation*}
		\partial_t u = \Delta\big(|u|^{m-1}u\big), \quad m \geq1,
	\end{equation*}
	it was shown in \cite{Pierre82} that for any initial datum  $(s,\zeta) \in \R_+\times \Pscr$, there is a unique weakly continuous distributional probability solution $u^{s,\zeta}$ in $\bigcap_{\tau > s, T>\tau}L^\infty((\tau,T)\times \R^d)$. In fact, it is shown that $u^{s,\zeta}$ is even $L^1$-continuous on $(s,\infty)$. Clearly, the uniqueness implies the flow property of the curves $t\mapsto u_t^{s,\zeta}(x)dx$. For $\zeta = \delta_{x_0}$, $u^{s,\zeta}$ is the \textit{Barenblatt solution} (see \cite{V07}), given by
	$$u_t^{s,\delta_{x_0}}(x) = (t-s)^{-\alpha}\bigg[\big(C-k|x-x_0|^2(t-s)^{-2\beta}\big)^+\bigg]^\frac
	{1}{m-1},\quad t >s, $$
	where $\alpha = \frac{d}{d(m-1)+2}, \beta = \frac \alpha d, k = \frac{\alpha(m-1)}{2md}$, $f^+ := f\vee 0$, and $C=C(m,d)>0$ is chosen such that $\int_{\R^d} u_t^{s,\zeta}(x) dx = 1$ for all $t >s$. The corresponding McKean-Vlasov equation is
	\begin{equation}\label{DDSDE-PME}
		dX_t = \sqrt{2u_t(X_t)^{m-1}}dB_t, \quad \mathcal{L}_{X_t} (dx)=u_t(x)dx, \quad t \geq s,\quad  \mathcal{L}_{X_s} = \zeta.
	\end{equation}
	
	By Lemma \ref{lem:Michael-observation}, the results of the second part of (i) above and Corollary \ref{cor:after-main-thm} with $\mathfrak{P}_0 = \Pscr_a^\infty$ and $\Pscr_0 = \Pscr$, there is a nonlinear Markov process $(\mathbb{P}_{s,\zeta})_{(s,\zeta) \in \R_+\times \Pscr}$ consisting of path laws $\mathbb{P}_{s,\zeta}$ of weak solutions to \eqref{DDSDE-PME} with one-dimensional time marginals given by $u_t^{s,\zeta}(x)dx$. For  $\zeta \in \Pscr_a^\infty$, $\mathbb{P}_{s,\zeta}$ is the path law of the unique weak solution to \eqref{DDSDE-PME} with one-dimensional time marginals in $\bigcap_{\tau>s, T>\tau}L^\infty((\tau,T)\times \R^d)$. This way, the Barenblatt solutions have a probabilistic interpretation as time marginal densities of a nonlinear Markov process. In the special case $d=1$, the identification of the Barenblatt solutions as one-dimensional time marginals of a stochastic process solving \eqref{DDSDE-PME} was already obtained in \cite{BCRV96}, but no connection to any kind of nonlinear Markov property was drawn.
	
	\item[(v)] \textbf{$2D$ vorticity Navier--Stokes equations.}
	Very recently, in \cite{BRZ23}, it was proven that Theorem \ref{theorem1} applies to the $2D$ vorticity Navier--Stokes equations, i.e. to 
	\begin{equation}\label{2DNSEvort}
		\partial_t u = \nu \Delta u - \divv(v u), \quad (t,x)\in \R_+\times \R^2,
	\end{equation}
	which is of type \eqref{NLFPE}, with $\Pscr_0 = \Pscr_a \cap L^4(\R^2)$. Here $\nu >0$ and $u = \curl v$. For details, including the McKean--Vlasov equation associated with \eqref{2DNSEvort}, we refer to \cite{BRZ23}.
	
	\item[(vi)] \textbf{Parabolic $p$-Laplace equation.}
	In the very recent paper \cite{BRR24-pLaplace}, the $p$-Laplace equation
	\begin{equation}\label{pLaplace}
		\partial_t u = \divv\big(|\nabla u|^{p-2} \nabla u\big),\quad (t,x)\in \R_+\times \R^d,
	\end{equation}
was reinterpreted as a nonlinear FPKE of type \eqref{intro:NLFPKE-Nemytskii} and it was proven there that our theory, more precisely Corollary \ref{cor:after-main-thm}, applies to the explicit fundamental solution (also called \textit{Barenblatt solution}) of \eqref{pLaplace}. This way, a uniquely determined nonlinear Markov process, consisting of solution path laws to the associated DDSDE and with one-dimensional time marginal densities given by this fundamental solution, was constructed. Due to its analogy to the case $p=2$, which is the classical heat equation and Brownian motion, in \cite{BRR24-pLaplace} this nonlinear Markov process was called $p$\textit{-Brownian motion}.
	

\end{enumerate}
\begin{rem}\label{rem:final-rem}
	With regard to Remark \ref{rem:nonlinear-feature} (i) we stress that in all the above examples, even in those where $\Pscr_0 = \Pscr$, it is not true that $p_{(s,\zeta),(r,y)} = \mathbb{P}_{r,\delta_y}$, since for the nonlinear equation \eqref{NLFPE} solutions are not stable under convex combinations in the initial datum. Hence, unlike for classical linear cases, for the nonlinear Markov processes from the previous examples one cannot expect to calculate the path laws $\mathbb{P}_{s,\zeta}$ via the one-dimensional time marginals $\mathbb{P}_{r,y}\circ (\pi^s_r)^{-1}$, $s\leq r$, $y \in \R^d$. Instead, one has the formula given in Proposition \ref{prop:Markov-distr-from-marginals}.
\end{rem}

\subsection{Ill-posed nonlinear FPKEs}\label{subsect:il-posed-appl}
Finally we point out again that the results from Section \ref{subsect:main-result} apply when no uniqueness for the nonlinear FPKE is known at all, but one can select a solution flow $\{\mu^{s,\zeta}\}_{(s,\zeta) \in \R_+\times \Pscr_0}$ for \eqref{NLFPE}, and if, additionally, one can establish the required restricted uniqueness for the linearized equations \eqref{LFPE}, with $\mu^{s,\zeta}_t$ replacing $\mu_t$, as formulated in Corollary \ref{cor:alternative-form-main-thm}. Solution flows for ill-posed nonlinear FPKEs have been constructed in \cite{Rehmeier-nonlinear-flow-JDE}. Since, in principle, (restricted) uniqueness for nonlinear equations is much harder to obtain than the mild restricted uniqueness condition for \eqref{LFPE}, it is often possible to construct nonlinear Markov processes in such nonlinear completely ill-posed cases, if one can select a solution flow. 

Let us give the following particular example in dimension $d=1$. In \cite{Scheutzow87}, it is shown that \eqref{DDSDE} with coefficients
\begin{equation*}
	\sigma = 0 , \quad b(t,\mu,x) = \int_{\R^d} h(y)\,\mu(dy)
\end{equation*}
has more than one solution for a specifically constructed function $h \in C_c(\R)$ and a suitable initial datum $\zeta \in \Pscr$. Consequently, since the solutions presented in \cite{Scheutzow87} have distinct one-dimensional time marginals, the corresponding nonlinear FPKE \eqref{NLFPE} with initial datum $\zeta$ is ill-posed as well. On the other hand, clearly Assumption B2 from \cite{Rehmeier-nonlinear-flow-JDE} holds for $\sigma$ and $b$ as above, and, therefore, by \cite[Prop.4.11, Rem.4.12]{Rehmeier-nonlinear-flow-JDE} there is a solution flow $\{\mu^{s,\zeta}\}_{(s,\zeta) \in \R_+\times \Pscr}$ for \eqref{NLFPE} with $\Pscr_0 = \Pscr$. Moreover, for each $(s,\zeta) \in \R_+\times \Pscr$, the corresponding linearized equation \eqref{LFPE}, with $\mu^{s,\zeta}_t$ in place of $\mu_t$, which is a continuity equation with bounded vector field only depending continuously on $t \geq s$, has $\mu^{s,\zeta}$ as its unique solution (in the sense of Definition \ref{def:all-single-eq.}). Consequently, by Corollary \ref{cor:alternative-form-main-thm}, there exists a nonlinear Markov process $(\mathbb{P}_{s,\zeta})_{(s,\zeta) \in \R_+\times \Pscr}$ such that the one-dimensional time marginals of $\mathbb{P}_{s,\zeta}$ are given by $(\mu^{s,\zeta}_t)_{t \geq s}$, and each $\mathbb{P}_{s,\zeta}$ is the path law of the unique weak solution of the corresponding DDSDE with these one-dimensional time marginals. It is straightforward to construct further examples of this prototype, e.g. with non-zero diffusion term and space-dependent drift $b$. 

\paragraph{Acknowledgement.} Funded by the Deutsche Forschungsgemeinschaft (DFG, German Research
Foundation) – Project-ID 317210226 – SFB 1283.
\bibliography{bib-collection}
\end{document}